\documentclass[11pt,letterpaper]{article}
\usepackage{amsmath,amssymb,amsfonts,amsthm,bbm}
\usepackage{epic,eepic,longtable}
\usepackage{multirow,verbatim}
\usepackage{array}
\usepackage{graphicx}
\usepackage{setspace}
\usepackage{url}
\usepackage{color}
\usepackage{subcaption}

\makeatletter
\def\singlespace{\def\baselinestretch{1}\@normalsize}

\makeatletter
\def\singlespace{\def\baselinestretch{1}\@normalsize}

\numberwithin{equation}{section}

\renewcommand{\hat}{\widehat}

\renewcommand{\hat}{\widehat}

\newcommand{\bfm}[1]{\ensuremath{\mathbf{#1}}}

     \def\CC{\mathbb{C}}
     
\def\be{\bfm e}     \def\EE{\mathbb{E}}

   \def\bI{\bfm I}

     \def\PP{\mathbb{P}}
     
\def\br{\bfm r}     \def\RR{\mathbb{R}}
     
\def\bt{\bfm t}

\def\bx{\bfm x}   \def\bX{\bfm X}  
\def\by{\bfm y}     
\def\bz{\bfm z}     

\def\calA{{\cal  A}} 
\def\calB{{\cal  B}}

\def\calE{{\cal  E}} 
\def\calF{{\cal  F}}

\def\calN{{\cal  N}}

\def\calS{{\cal  S}}

\newcommand{\bfsym}[1]{\ensuremath{\boldsymbol{#1}}}

 \def\bbeta{\bfsym \beta}
              
 \def\bdelta{\bfsym {\delta}}

 \def\btheta{\bfsym {\theta}}           \def\bTheta {\bfsym {\Theta}}
           \def\bepsilon{\bfsym \varepsilon}
              \def\bSigma{\bfsym \Sigma}

 \def\bxi{\bfsym {\xi}}
 \def\bzeta{\bfsym {\zeta}}
 
\def\newpage{\vfill\eject}

\def\today{\ifcase\month\or
  January\or February\or March\or April\or May\or June\or
  July\or August\or September\or October\or November\or December\fi
  \space\number\day, \number\year}

\newdimen\biblioindent    \biblioindent=30pt

 at 8truept

\newcommand{\beq}{\begin{equation}}
  \newcommand{\eeq}{\end{equation}}
\newcommand{\beqn}{\begin{eqnarray}}
  \newcommand{\eeqn}{\end{eqnarray}}
\newcommand{\beqnn}{\begin{eqnarray*}}
  \newcommand{\eeqnn}{\end{eqnarray*}}

\allowdisplaybreaks
\usepackage{verbatim}
\renewcommand{\baselinestretch}{1.66}
\newtheorem{dfn}{Definition}

\newcounter{CondCounter}

\numberwithin{equation}{section}
\theoremstyle{plain}
\newtheorem{theorem}{Theorem}[section]

\theoremstyle{definition}
\newtheorem{remark}{Remark}

\theoremstyle{definition}

\theoremstyle{definition}
\newtheorem{lemma}{Lemma}

\newtheorem{corollary}{Corollary}[section]

\theoremstyle{definition}

\theoremstyle{definition}

\theoremstyle{definition}

\usepackage{mathtools}

\begin{document}
\title{A Fourier Analytical Approach to Estimation of Smooth Functions in Gaussian Shift Model}
\author{Fan Zhou and Ping Li\\
Baidu Research\\
10900 NE 8th St. Bellevue, WA 98004, USA\\
\{fanzhou,\ liping11\}@baidu.com
}
\date{}
\maketitle

\begin{abstract}
\noindent Let $\bx_j = \btheta + \bepsilon_j$, $j=1,\dots,n$ be i.i.d. copies of a Gaussian random vector $\bx\sim\calN(\btheta,\bSigma)$ with unknown mean $\btheta \in \RR^d$ and unknown covariance matrix $\bSigma\in \RR^{d\times d}$. The goal of this article is to study the
estimation of $f(\btheta)$ where $f$ is a given smooth function of which smoothness is characterized by a Besov-type norm.
The problem of interest resides in the high dimensional regime where the intrinsic dimension can grow with the sample size $n$.
Inspired by the classical work of A. N. Kolmogorov~\cite{kolmogorov1950unbiased} on unbiased estimation and Littlewood-Paley theory, we develop a new estimator based on a Fourier analytical approach that achieves effective bias reduction.
Asymptotic normality and efficiency are proved when the smoothness index of $f$ is above certain threshold which was discovered
recently by~\cite{koltchinskii2018efficient} for a H\"{o}lder type class. Numerical simulations are presented to validate our analysis. The simplicity of implementation and its superiority over the plug-in approach indicate
the new estimator can be applied to a broad range of real~world~applications.
\end{abstract}

\newpage

\section{Introduction}
\label{section: introduction}
Let
\begin{equation}
\label{model}
	\bx_j = \btheta + \bepsilon_j,~j=1,...,n
\end{equation}
be the Gaussian shift model with $\bx_j$ being i.i.d. noisy observations of an unknown parameter $\btheta \in \RR^d$, and $ \bepsilon_j \in \RR^d$ being i.i.d. copies of a non-degenerate Gaussian random noise vector $\bepsilon\sim \calN(0,\bSigma)$. The goal of this article is to study the problem of estimation of $f(\btheta)$ for a given smooth function $f: \RR^d \rightarrow \RR$.
$\bepsilon_j$'s can be treated as the measurement error when one observes $\btheta$. This indicates that the problem we study appears almost everywhere in real world applications.
Particularly, we are interested in the regime where the intrinsic dimension $\br(\bSigma)$ of the parameter can grow with the sample size.
To be more specific, we take the sample mean $\bar{\bx}= n^{-1}\sum_{j=1}^n \bx_j$ which is a sufficient statistic used to estimate $\btheta$ and then the model (\ref{model}) can be equivalently written as a form of $\bar{\bx}=\btheta + \bar{\bepsilon}$ with $\bar{\bepsilon} \sim \mathcal{N}(\mathbf{0}; n^{-1}\bSigma)$. In this article, we assume that $\|\bSigma\|_{op}= O(1)$, $\br(\bSigma) = n^{\alpha}$ with $\alpha\in (0,1)$ which grows as the sample size $n\rightarrow \infty$.

Studies under the setting where the dimension of the underlying parameter is allowed to grow with the sample size can be traced back to~\cite{portnoy1984asymptotic,portnoy1985asymptotic,portnoy1986asymptotic,mammen1989}.
Early results in the study of efficient estimation of smooth functionals are mostly focused on infinite-dimensional parameter space where people were trying to build the connection between the geometric complexity of the parameter space and the modulus of continuity of the functional.
Notable results include but not limited to~\cite{levit1976efficiency,levit1978asymptotically,ibragimov2013statistical,ibragimov1986some,bickel1988estimating,nemirovskii1991necessary,nemirovski2000topics,birge1995estimation,laurent1996efficient,lepski1999estimation}. Two types of special functionals are extensively studied. Results on estimation of linear functionals include~\cite{donoho1987minimax,donoho1991geometrizing,cai2005adaptive,klemela2001sharp} and the references therein.
Results in terms of estimation of quadratic functionals include~\cite{donoho1990minimax,cai2005nonquadratic,klemela2006sharp,laurent2000adaptive,bickel2003} and the references therein. Recent results show a surge of interest
in efficient and minimax optimal estimation of functionals of parameter in high dimensional models or models with growing dimension, see~\cite{collier2017minimax,van2014asymptotically,doi:10.1111/rssb.12026,koltchinskii2018efficient}.
Recent works~\cite{han2020estimation,goldenshluger2020minimax1,goldenshluger2020minimax2,collier2017minimax,wu2016minimax,wu2019chebyshev} that made impressive progress on minimax estimation of non-smooth functionals are also  interesting.

\newpage

We consider a given smooth function $f$ whose smoothness is characterized by a Besov-type norm $\|\cdot\|_{s,\infty,1}$.  Based on the fruitful idea of Littlewood-Paley theory and the seminal work by A. N. Kolmogorov~\cite{kolmogorov1950unbiased} on unbiased estimation, we construct a new estimator in Section~\ref{section: estimator} via a Fourier analytical approach and define it as
\begin{equation}
\label{estimator: intro}
g(\bar{\bx}) :=   \frac{1}{(2\pi)^{d/2}}   \int_{\Omega}  \mathcal{F} f(\bzeta)e^{\langle \bSigma \bzeta, \bzeta \rangle/2n} e^{i\bzeta \cdot \bar{\bx}} d\bzeta,
\end{equation}
where $i=\sqrt{-1}$, $\mathcal{F}f$ denotes the Fourier transformation of $f$, and $\Omega \subset \{ \bzeta \in \RR^d : \|\bzeta\|\leq R \}$ is a truncated region in the frequency domain.
The new estimator is easy to implement and can be widely used in practice since it deals with Fourier transform data.
For instance, Fourier transform is used widely throughout medical imaging where its applications include:
determining the spatial resolution of imaging systems, spatial localisation in magnetic resonance imaging, analysis of Doppler ultrasound signals,
and image filtering in emission and transmission computed tomography.
An immediate implication following the construction of $g(\bar{\bx})$ is that $g(\bar{\bx})$ is an unbiased estimator
of $f(\btheta)$ when $f$ is an entire function of exponential type. In Section~\ref{section:bias}, we show that for a general $f$ with smoothness index $s$, the bias of $g(\bar{\bx})$ is bounded by
\begin{equation}
\label{rate: bias}
	|\EE_{\btheta} g(\bar{\bx}) - f(\btheta)| \lesssim \|f\|_{s,\infty,1}\cdot R^{-s}.
\end{equation}
With an adaptively chosen $R$ based on different size of $\|f\|_{s,\infty,1}$, the bias can be controlled of a smaller order than $O(n^{-1/2})$ when $s>1/(1-\alpha)$.
Note that a recent series of works by~\cite{koltchinskii2018asymptotically,koltchinskii2019estimation} considered similar problems via a different approach. Specifically, they developed an innovative method through
iterative bootstrap to achieve bias reduction for a H\"{o}lder type class. Another work~\cite{jiao2017bias} used the similar approach as in~\cite{koltchinskii2018asymptotically} and~\cite{koltchinskii2018efficient} to study the estimation of smooth function of parameter of binomial model.

\newpage

In Section~\ref{section: normal}, we show that when $n^{-1/2}\sqrt{\langle \bSigma \nabla f(\btheta) , \nabla f(\btheta) \rangle} \asymp \|\bSigma\|_{op}\|f\|_{s,\infty,1}/\sqrt{n}$, the following asymptotic normality holds for the proposed estimator:
\begin{equation}
	\label{rate: variance}
	\frac{\sqrt{n}\big( g(\bar{\bx}) -f(\btheta)\big) }{  \sqrt{\langle \bSigma \nabla f(\btheta) , \nabla f(\btheta) \rangle} }  \Rightarrow \calN(0,1),~~{\rm as}~~n\rightarrow \infty,
\end{equation}
where $\calN(0,1)$ is the standard normal random variable and $ n^{-1}\cdot \langle \bSigma \nabla f(\btheta) , \nabla f(\btheta) \rangle$ is the reciprocal of the Fisher information for the estimation of $f(\btheta)$ based on the observation $\bx \sim \calN(\btheta, n^{-1}\bSigma)$.
It means that under mild restrictions $g(\bar{\bx})$ is normally distributed around the true parameter
$f(\btheta)$. Such results are critical in real world applications and provide us theoretical guarantees of building reliable confidence intervals containing the true parameter using the estimator. Together with (\ref{rate: bias}), (\ref{rate: variance})
implies that
\begin{equation}
	\label{rate: mse}
	\EE_{\btheta} \big( g(\bar{\bx}) - f(\btheta) \big)^2 \lesssim \big( \|f\|^2_{s,\infty,1}\cdot R^{-2s} \vee \|f\|^2_{s,\infty,1}\cdot n^{-1} \big).
\end{equation}

In Section~\ref{section: lower}, we establish several lower bounds. Especially, according to the lower bounds, the variance $ n^{-1}\cdot \langle \bSigma \nabla f(\btheta) , \nabla f(\btheta) \rangle$ in asymptotic normality is optimal when $n^{-1/2}\sqrt{\langle \bSigma \nabla f(\btheta) , \nabla f(\btheta) \rangle} \asymp \|\bSigma\|_{op}\|f\|_{s,\infty,1}/\sqrt{n}$. This essentially means that the proposed estimator $g(\bar{\bx})$ is asymptotically efficient.
Moreover, we show that if $\|f\|^2_{s,\infty,1} =O(1)$, the proposed estimator is minimax optimal with a
properly chosen $R$ under standard Gaussian shift model. Such minimax rates imply a sharp threshold
on smoothness which was recently discovered by~\cite{koltchinskii2018efficient} when studying a H\"{o}lder type class.
The proof of the lower bounds are based on some new construction and ideas which are different from the existing results.
Early results of this kind such as~\cite{ibragimov1986some,nemirovskii1991necessary,nemirovski2000topics} established the threshold on smoothness in terms of Kolmogorov widths which characterize the complexity of the parameter space.

In Section~\ref{section: adaptive}, we propose a data-driven estimator to address the adaptation issue when the covariance matrix $\bSigma$  in (\ref{estimator: intro}) is unknown and show that the difference between the adaptive estimator and
$g(\bar{\bx})$ is asymptotically negligible with high probability. Numerical simulation results are presented in Section~\ref{section: simulation} to validate our theory, showing that the new estimator's performance
is superior to its plug-in counterpart on both bias and variance reduction and building reliable confidence intervals.

\section{Preliminaries and Notations}
\label{section: prelim}
\subsection{Notations}

We use boldface uppercase letter $\bX$ to denote a matrix and boldface lowercase letter $\bx$ to denote a vector.
We use $\|\cdot\|$ to denote the $\ell_2$-norm of a vector, $\|\cdot\|_{op}$ to denote the spectral norm (largest singular value) of a matrix, and $\|\cdot\|_p$ to denote the $L^p$-norm of a function.
For a covariance matrix $\bSigma$, we use $\br(\bSigma): = {\rm tr}(\bSigma)/\|\bSigma\|_{op}$ to denote the effective rank (intrinsic dimension) of $\bSigma$. In the rest of this article, we assume
$\| \bSigma\|_{op} = O(1)$ which does not involve the sample size parameter $n$.
 We use $\calS = \calS(\RR^d)$ to denote the Schwartz space and $\calS'=\calS'(\RR^d)$ to denote the set of tempered distributions on $\RR^d$.
 We use $\mathcal{F}$ and $\mathcal{F}^{-1}$ to denote the Fourier transform (FT) and inverse Fourier transform (IFT) respectively. Throughout the paper, given nonnegative $a$ and $b$, $a\lesssim b$ means
that $a\leq Cb$ for a numerical constant $C$, and $a \asymp b$ means that $a \lesssim b$ and $b \lesssim a$. $a \wedge b = \min \{ a, b\}$ and $a \vee b = \max \{ a, b\}$.

\subsection{Besov-Type Norm and Entire Function of Exponential Type}
\label{function}
In this section, firstly we define a Besov-type norm of our interest.
Given $f:\RR^d \rightarrow \RR$, we define the following norm $\|\cdot \|_{s,\infty, 1}$ of $f$:
\begin{equation}
	\label{norm_f}
	\big\| f \big\|_{s,\infty, 1} :=  (2\pi)^{-d/2} \int \big| \calF f(\bzeta)\big|\big( 1 \vee \| \bzeta\|^s \big) d\bzeta.
\end{equation}
Especially, we denote by
\begin{equation}
	\calF^s(\RR^d) := \big\{ f: \big\| f \big\|_{s,\infty, 1} < \infty \big\}.
\end{equation}
Note that the parameter $s$ in (\ref{norm_f}) characterizes the smoothness of $f$. It is defined in a similar fashion as the smoothness parameter of Besov norm via Littlewood-Paley decomposition, see Section 2.3.1 in~\cite{Triebelfunction}. One should note that the above norm
$\big\| f \big\|_{s,\infty, 1}$ can easily depend on the dimension parameter $d$. One important example of such function is when
\begin{equation}
\label{mixture}
f(\btheta) = \int e^{-\|\btheta - \by\|^2/2} dG(\by)
\end{equation}
where $G(\by)$ is some distribution function in $\RR^d$. In this case,
\begin{equation}
	\big\| f \big\|_{s,\infty, 1} = (2\pi)^{-d/2} \int e^{-\|\bzeta\|^2/2}  \big( 1 \vee \| \bzeta\|^s \big) d\bzeta \leq \big( 1+o(1) \big)d^{s/2}.
\end{equation}

Smooth functions whose Fourier transform has compact support is central in constructing the estimator for our problem. These functions are closely related to the entire functions of exponential type. We introduce the definition as follows.
\begin{dfn}
	Let $f: \CC^d \rightarrow \CC$ be an entire function and $\sigma:=(\sigma_1,...,\sigma_d)$, $\sigma_j>0$. Function $f$ is of exponential type $\sigma$ if for any
	$\varepsilon >0 $ there exists a constant $C(\varepsilon,\sigma,f)>0$ such that
	\begin{equation}
		\big| f(\bz) \big| \leq C(\varepsilon,\sigma,f) e^{\sum_{j=1}^d (\sigma_j+\varepsilon)|z_j|},~\forall~\bz \in \CC^d.
	\end{equation}
\end{dfn}

The following theorem is part of the well-known Paley-Wiener-Schwartz theorem which offers an necessary and sufficient condition for us to identify entire functions of exponential type. We refer to Theorem 1.7.7 in~\cite{hormander2015analysis} for a more detailed discussion in case the reader is interested.
\begin{theorem}
	\label{paley-wiener}
	The following two assertions are equivalent:
	\begin{enumerate}
		\item $\varphi \in \mathcal{S}'$ and ${\rm supp}(\mathcal{F}\varphi)\subset \{\bx:\|\bx\|\leq \sigma \}$ is bounded;
		\item $\varphi(\bz)$ for all $\bz \in \CC^d$ is an entire function of exponential type $\sigma$.
	\end{enumerate}
\end{theorem}

\section{Bias Reduction and Estimator Construction}
\label{section: estimator}
In this section, we introduce a bias reducing estimator based on a Fourier analytical approach and the Gaussian kernel. The origin of this idea can be traced back to A.N. Kolmogorov~\cite{kolmogorov1950unbiased} in which the author tried to build the connection between unbiased estimation in Gaussian shift model with ``the inverse heat conductivity problem".
The intuition that lies behind the construction of the estimator is pretty straight forward. To find a good estimator $g(\bar{\bx})$ of $f(\btheta)$ with small bias depends on how well one can solve the following
integral equation
\begin{equation}
\label{equation: integral}
	\EE_{\btheta} g(\bar{\bx}) = f(\btheta).
\end{equation}
Instead of solving it directly which can be hard, we approximately solve it by replacing the right hand side with some good approximation of $f$. The proxy of $f$ is chosen via a proper truncation on $f$'s frequency domain based on the spirit of well-known Littlewood-Paley decomposition. A different approach based on a bootstrap chain bias reduction technique was developed recently by~\cite{koltchinskii2018asymptotically,koltchinskii2018efficient} to approximately solve (\ref{equation: integral}) which also achieves effective bias reduction. However, the implementation of this method can be quite difficult since such estimators oftentimes do not have an explicit form and to compute its approximate surrogate can be computationally intensive.
The bootstrap chain bias reduction technique is also used in~\cite{jiao2017bias} to estimate smooth functions of the parameter of binomial models.

To start with, we review some of the basic knowledge from PDE, and harmonic analysis. Given $\bx \sim \mathcal{N}(\btheta; \bSigma)$.
We denote the density function of $\bx$ by
$$
p(\bx|\btheta,\bSigma) := \frac{1}{\sqrt{(2\pi)^d  |\bSigma|}}\exp\Big\{  - \frac{1}{2} \langle \bSigma^{-1}(\bx-\btheta), \bx-\btheta\rangle \Big\},
$$
where $|\bSigma|$ denotes the determinant of $\bSigma$.
For any given estimator $g(\bar{\bx})$ of $f(\btheta)$, it is easy to check that
\begin{equation}
\label{expectation}
\mathbb{E}_{\btheta}g(\bar{\bx}) = \mathbb{E}_{\btheta}g(\btheta+\bar{\bepsilon}) =  \mathbb{E}_{\btheta}g(\btheta - \bar{\bepsilon}) =  \int_{\mathbb{R}^d} g(\btheta - \bzeta) p(\bzeta|{\bf 0}; n^{-1}\bSigma) d\bzeta.
\end{equation}
Note that the right hand side of (\ref{expectation}) is the convolution of $g$ with a Gaussian density $p$ with zero mean and covariance matrix $\bSigma$. We denote by $h$ this convolution
\begin{equation}
\label{convolution}
h(\btheta) := g \ast p^o (\btheta) =  \int_{\mathbb{R}^d} g(\btheta - \bzeta) p(\bzeta|{\bf 0}; n^{-1}\bSigma) d\bzeta,
\end{equation}
where $p^o(\bzeta):=p(\bzeta|{\bf 0}; n^{-1}\bSigma) $.
Recall that the Fourier Transform of a function $f:\mathbb{R}^d\rightarrow \mathbb{R}$ is defined as
$$
\mathcal{F}f(\bzeta) := \frac{1}{(2\pi)^{d/2}}\int_{\mathbb{R}^d} f(\btheta) e^{-i \bzeta\cdot \btheta}d\bx.
$$
Given (\ref{convolution}), the basic properties of Fourier transform and convolution (see~\cite{reed1975ii}, page 6) lead to
\begin{equation}
\label{conv2}
\mathcal{F} h = \mathcal{F} (g \ast p^o) = (2\pi)^{d/2} \mathcal{F}g \cdot \mathcal{F}p^o.
\end{equation}
It is easy to see that
$$
\mathcal{F} p^o(\bzeta)  = \frac{1}{(2\pi)^{d/2}} e^{-\frac{1}{2n}\langle \bSigma \bzeta, \bzeta \rangle}.
$$
Thus, from (\ref{conv2}) we have
$$
\mathcal{F} g (\bzeta) =\mathcal{F} h(\bzeta)e^{\langle \bSigma \bzeta, \bzeta \rangle/2n}.
$$
Now we take the Inverse Fourier Transform of $\mathcal{F}g(\bzeta)$ and obtain our estimator
\begin{equation}
\label{estimator: general}
g(\bar{\bx}) =\mathcal{F}^{-1} (\mathcal{F} g) =  \frac{1}{(2\pi)^{d/2}} \int_{\mathbb{R}^d} \mathcal{F} g (\bzeta) e^{ i\bzeta\cdot \bar{\bx}}d\bzeta =  \frac{1}{(2\pi)^{d/2}}   \int_{\mathbb{R}^d}  \mathcal{F} h(\bzeta)e^{\langle \bSigma \bzeta, \bzeta \rangle/2n} e^{i\bzeta \cdot \bar{\bx}} d\bzeta .
\end{equation}
One should note that the integral in (\ref{estimator: general}) can be meaningless when the integral on the right hand side diverges. Indeed, the term $e^{\langle \bSigma \bzeta, \bzeta \rangle/2n}$ inside the integral grows exponentially fast when $\|\bzeta\|$ goes to infinity. If $|\mathcal{F} h(\bzeta)|$ does not decay fast enough as $\|\bzeta\|$ goes to infinity, then $g(\bar{\bx})$ in (\ref{estimator: general}) may not be well-defined.

We consider a given function $f\in \calF^s(\RR^d)$. We use the following function $f^N$ to approximate $f$. We denote by
\begin{equation}
f^N(\btheta) := (2\pi)^{-d/2} \int_{\Omega} \calF f(\bzeta)\cdot e^{i\bzeta^T \btheta}  d\bzeta
\end{equation}
with $\Omega:= \{ \bzeta: \|\bzeta\| \leq R\}$ for some $0<R<\infty$. Since $f\in\calF^s(\RR^d)$, then it is easy to check that $f^N$ is well-defined.
With $R$ being finite, $f^N$ is an analytic function based on Paley-Wiener-Schwartz theorem (Theorem~\ref{paley-wiener}). Now we formally introduce the estimator
$g(\bar{\bx})$ of $f(\btheta)$ under model (\ref{model}) as the following:
\begin{equation}
\label{estimator: f}
g(\bar{\bx}) :=   \frac{1}{(2\pi)^{d/2}}   \int_{\Omega}  \mathcal{F} f^N(\bzeta)e^{\langle \bSigma \bzeta, \bzeta \rangle/2n} e^{i\bzeta \cdot \bar{\bx}} d\bzeta.
\end{equation}

An immediate implication of the above analysis is that when $h$ is an entire function of exponential type, then
due to an extension of Paley-Wiener Theorem to $\RR^d$ by E.M. Stein~\cite{stein1957functions}, $g(\bar{\bx})$ defined in (\ref{estimator: general}) is an unbiased estimator of $f=h$ under model (\ref{model}). We summarize this in the following theorem.

\begin{theorem}
	\label{theorem: unbias}
		Under model (\ref{model}), let $h:\CC^d\rightarrow \CC$ be an entire function of exponential-type $\sigma$.
		Then the estimator $g(\bar{\bx})$ defined in (\ref{estimator: general}) is an unbiased estimator of $h(\btheta)$.
\end{theorem}

\section{Bound on the Bias}
\label{section:bias}

In this section, we derive an upper bound on the bias of the estimator (\ref{estimator: f}) for a given $f\in \calF^s(\RR^d)$.
We show that the upper bound on the bias $\EE_{\btheta}g(\bar{\bx}) - f(\btheta)$ is characterized by its norm $\|f\|_{s,\infty,1}$, the smoothness parameter $s$, and the truncation radius $R$ in the frequency domain.

\begin{theorem}
	\label{theorem: bias}
	Under model (\ref{model}), assume that given $f\in \calF^s(\RR^d)$ with $s\geq 0$ and the estimator $g(\bar{\bx})$ defined as in (\ref{estimator: f}), we denote by $R:=\sup_{\bzeta\in \Omega}\|\bzeta\|$.
	Then the following bound on the bias holds:
	\begin{equation}
	\label{bound: bias}
		\big| \EE_{\theta} g(\bar{\bx}) - f(\btheta)\big| \leq  \|f\|_{s,\infty,1} \cdot R^{-s}.
	\end{equation}
	Especially, when $ \|f\|_{s,\infty,1}  \leq 1$, by taking $\br(\bSigma) = O(n^{\alpha})$ and $R = \sqrt{n/\br(\bSigma)}$, we have
	\begin{equation}
		\big| \EE_{\theta} g(\bar{\bx}) - f(\btheta)\big| \leq  n^{s\cdot (\alpha-1)/2}.
	\end{equation}
\end{theorem}

\begin{remark}
	Theorem~\ref{theorem: bias} shows that higher order smoothness (larger $s$) can contribute to better bias reduction. Especially, when $\big\|f\big\|_{s,\infty,1}  \leq 1$ and the intrinsic dimension of the parameter $\br(\bSigma) = O(n^{\alpha})$ for some $\alpha \in (0,1)$, to make $g(\bar{\bx})$ be a $\sqrt{n}$-consistent estimator in order to achieve asymptotic normality and efficiency, one needs $s > 1/(1-\alpha)$. Such threshold on smoothness was proved to be sharp by a recent work~\cite{koltchinskii2018efficient} over a H\"older-type space for Gaussian shift model.
	\end{remark}
	
  Interestingly, when $\big\|f\big\|_{s,\infty,1} = O(d^{s/2})$, namely the size of $f$ can depend on the dimension parameter $d$ such as in the mixture model (\ref{mixture}), one can choose a larger truncation parameter $R \asymp \sqrt{n}$ in order to achieve a similar rate of $O(n^{s\cdot (\alpha-1)/2})$ on bias.
  \begin{corollary}
  	\label{corollary: bias_d}
  	Under the same condition of Theorem~\ref{theorem: bias}, assume that $\big\|f\big\|_{s,\infty,1} = O(d^{s/2})$ with $s\geq 0$ and $d = O(n^{\alpha})$, and the estimator $g(\bar{\bx})$ defined in (\ref{estimator: f}) with $R = \sqrt{n}$.
  	Then the following bound on the bias holds with some constant $C_1$:
  	\begin{equation}
  		\label{bound: bias_d}
  		\big| \EE_{\theta} g(\bar{\bx}) - f(\btheta)\big| \leq C_1 n^{s\cdot (\alpha-1)/2}.
  	\end{equation}
  \end{corollary}

\begin{remark}
	Corollary~\ref{corollary: bias_d} shows that even when the norm $\big\|f\big\|_{s,\infty,1}$ is large and depends on $d$, the bias of the proposed estimator $g(\bar{\bx})$ can still be well controlled below $O(n^{-1/2})$ as long as $s>1/(1-\alpha)$.
\end{remark}

\section{Asymptotic Normality and Efficiency}
\label{section: normal}

In this section, we show the asymptotic normality of the estimator $g(\bar{\bx})$. Such results are of vital importance in terms of evaluation of an estimator in statistical inference both in theory and in practice. It can used to justify whether we can use the estimator to build reliable confidence intervals for estimation of the true parameter.

For a given function $f:\RR^d \rightarrow \RR$, we denote by
\begin{equation}
	\sigma^2_{f,\bepsilon}(\btheta):= n^{-1}\cdot \langle \bSigma \nabla f(\btheta) , \nabla f(\btheta) \rangle .
\end{equation}
Note that $\sigma^2_{f,\bepsilon}(\btheta)$ is the reciprocal of the Fisher information for the estimation of $f(\btheta)$ based on the observation
$\bar{\bx} \sim \calN(\btheta, n^{-1}\bSigma)$. According to the definition of $\|f\|_{s,\infty,1}$, for $s\geq 1$ term $\sigma_{f,\bepsilon}(\btheta)$
is naturally bounded by
\begin{equation}
	\sigma_{f,\bepsilon}(\btheta) \leq \|\bSigma\|_{op}\|f\|_{s,\infty,1}/\sqrt{n}.
\end{equation}
As we shall see in the following theorem, when $\sigma_{f,\bepsilon}(\btheta) \asymp \|\bSigma\|_{op}\|f\|_{s,\infty,1}/\sqrt{n}$, the estimator $g(\bar{\bx})$ is normally distribution around the true parameter $f(\btheta)$ with asymptotic variance $\sigma^2_{f,\bepsilon}(\btheta)$.
and moreover, such variance is optimal.
Recall that in Theorem~\ref{theorem: bias} we showed $\big| \EE_{\btheta}g(\bar{\bx}) - f(\btheta)\big|$ is uniformly of a small order when $s>1/(1-\alpha)$. As it turns out, such threshold on smoothness is also essential for us to establish asymptotic normality of the proposed estimator $g(\bar{\bx})$. As we shall see in Section~\ref{section: simulation}, our simulation results show that the confidence intervals built based on the estimator $g(\bar{\bx})$ is much more reliable
than those built based on the plug-in estimator $f(\bar{\bx})$.

\begin{theorem}
	\label{theorem: normal}
	Under model (\ref{model}), assume that for a given $f\in \calF^s(\RR^d)$ with $s>1$ and for some absolute constant $\tau>0$
	\begin{equation}
	\label{variance: size}
	  \frac{n^{-1/2}\big\| f \big\|_{s,\infty,1}\|\bSigma\|^{1/2}_{op}}{\sigma_{f,\bepsilon}(\btheta)} \leq \tau
	\end{equation}
	Take $\br(\bSigma) = O(n^{\alpha})$ with $a\in(0,1)$ and $R \leq \sqrt{n}$. Then if $s>1/(1-\alpha)$
	\begin{equation}
		\frac{g(\bar{\bx}) -f(\btheta)}{\sigma_{f,\bepsilon}(\btheta)} = \frac{\sqrt{n}\big( g(\bar{\bx}) -f(\btheta)\big) }{  \sqrt{\langle \bSigma \nabla f(\btheta) , \nabla f(\btheta) \rangle} }  \Rightarrow \calN(0,1),~~{\rm as}~~n\rightarrow \infty.
	\end{equation}
	where $\calN(0,1)$ is the standard normal random variable.
	Moreover,
	\begin{equation}
	\label{bound: efficiency}
	\begin{aligned}
	&\frac{ \EE_{\btheta}^{1/2} (g(\bar{\bx}) - f(\btheta))^2}{\sigma_{f,\bepsilon}(\btheta)} \rightarrow 1,~~{\rm as}~~n\rightarrow \infty.
	\end{aligned}
	\end{equation}
\end{theorem}

\begin{remark}
	Note that Theorem~\ref{theorem: normal} holds with different choices of $\|f\|_{s,\infty,1}$. For instance if $\|f\|_{s,\infty,1} = O(1)$ one can choose $R=\sqrt{n/\br(\bSigma)}$ and if $\|f\|_{s,\infty,1} = O(d^{s/2})$ one can choose $R=\sqrt{n}$. However,
	different choice of $\|f\|_{s,\infty,1}$ leads to different variance $\sigma^2_{f,\bepsilon}(\btheta)$ since condition (\ref{variance: size}) indicates $\sigma_{f,\bepsilon}(\btheta) \asymp \|\bSigma\|_{op}\|f\|_{s,\infty,1}/\sqrt{n}$. Indeed, together with Theorem
	\ref{theorem: bias}, Theorem~\ref{theorem: normal} shows that
	\begin{equation}
		\label{bound: MSE}
		 \EE_{\btheta}\big(g(\bar{\bx}) - f(\btheta)\big)^2 \lesssim \|f\|^2_{s,\infty,1}\cdot n^{-1} + \|f\|^2_{s,\infty,1}\cdot R^{-2s}.
	\end{equation}
    When $s>1/(1-\alpha)$, we have shown that the bias term $ \|f\|^2_{s,\infty,1}\cdot R^{-2s}$ can always be of a smaller order than $O(n^{-1})$ as long as $\|f\|_{s,\infty,1} = O(d^{s/2})$. However, the variance $\|f\|^2_{s,\infty,1}\cdot n^{-1}$ can be much larger than $O(n^{-1})$ when $\|f\|_{s,\infty,1}$ is large, say $\|f\|_{s,\infty,1} = O(d^{s/2})$ and $d = n^{\alpha}$ with $\alpha\in(0,1)$. As we shall see in Section~\ref{section: lower}, when
    (\ref{variance: size}) holds, $\sigma^2_{f,\bepsilon}(\btheta)$ is the optimal variance, and when $\|f\|_{s,\infty,1} = O(1)$, (\ref{bound: MSE}) implies the minimax optimal rate of MSE, namely $O\big(n^{-1} \vee (d/n)^{s}\big)$ under standard Gaussian shift
    model.
\end{remark}

\section{Lower Bounds}
\label{section: lower}

In this section, we establish several lower bounds under model (\ref{model}).
In Theorem~\ref{theorem: minimax_lower_1} and Theorem~\ref{theorem: minimax_lower_2} below, we show two minimax lower bounds under standard Gaussian shift model, namely $\bSigma = \bI_d$.
In this case, the intrinsic dimension $\br(\bSigma) = d$. These two theorems together show that for a given $f\in \calF^s(\RR^d)$ with $\|f\|_{s,\infty,1}\leq 1$, the proposed estimator $g(\bar{\bx})$
is minimax optimal. Our methods to prove these results are based on some new techniques which are quite different from the previous methods introduced by~\cite{nemirovskii1991necessary,nemirovski2000topics}.
Note that a recent result in~\cite{koltchinskii2018efficient} attained a similar type of minimax lower bounds for a special H\"older type function class with different methods.
Next, we prove another lower bound in Theorem~\ref{theorem: efficiency} which essentially shows that the asymptotic variance in Theorem~\ref{theorem: normal} is optimal which implies
asymptotic efficiency of the proposed estimator. The method we use to prove this lower bound is based on an application of van Trees inequality introduced by~\cite{koltchinskii2018efficient}.

\begin{theorem}
	\label{theorem: minimax_lower_2}
	Assume that $f \in \calF^s(\RR^d)$ such that $\|f\|_{s,\infty,1}\leq B$ (B can be as large as $d^{s}$) with $s > 1$, and under model (\ref{model}), $\bar{\bx} \sim \mathcal{N}(\btheta; n^{-1}\bI_d)$.
	Then with some absolute constant $c_1$, the following lower bound holds
	\begin{equation}
	\label{bound: lower_2}
	\inf_T \sup\limits_{\|f\|_{s,\infty,1}\leq B}\sup\limits_{\|\btheta\|\leq 1} \EE_{\btheta} (T(\bar{\bx}) - f(\btheta))^2 \geq c_1 \big(d/n \big)^s
	\end{equation}
\end{theorem}

Now we switch to prove another minimax lower bound. The proof is based on an application of the well-known Assouad's Lemma (\cite{Tsybakov:2008:INE:1522486} Lemma 2.12).

\begin{theorem}
	\label{theorem: minimax_lower_1}
	Assume that $f \in \calF^s(\RR^d)$ such that $\|f\|_{s,\infty,1}\leq B$ (B can be as large as $d^{s}$) with $s > 1$, and under model (\ref{model}), $\bar{\bx} \sim \mathcal{N}(\btheta; n^{-1}\bI_d)$.  Then for some numerical constant $c'_1>0$
	\begin{equation}
	\label{bound: lower_1}
	\inf_T \sup\limits_{\|f\|_{s,\infty,1}\leq B}\sup\limits_{\|\btheta\|\leq 1} \EE_{\btheta} (T(\bar{\bx}) - f(\btheta))^2 \geq c'_1 n^{-1}.
	\end{equation}
\end{theorem}

\begin{remark}
	Combining the bounds in (\ref{bound: lower_1}) and (\ref{bound: lower_2}), we obtain when $d = n^{\alpha}$ with $\alpha \in (0,1)$
	\begin{equation}
    \inf_T \sup\limits_{\|f\|_{s,\infty,1}\leq B}\sup\limits_{\|\btheta\|\leq 1} \EE_{\btheta} (T(\bx) - f(\btheta))^2
	\gtrsim \big( n^{-1} \vee n^{-s(1-\alpha)} \big) .
	\end{equation}
    Note that when $\|f\|_{s,\infty,1}=O(1)$, this matches the upper bound in (\ref{bound: MSE}) which shows $g(\bar{\bx})$ is minimax optimal for standard Gaussian shift model, which was proved previously
    by~\cite{koltchinskii2018efficient} for a H\"{o}lder-type function class.
    However,
    when $\|f\|_{s,\infty,1}$ depends on $d$, there is a gap between this lower bound and the upper bound (\ref{bound: MSE}). Because the term $\|f\|^2_{s,\infty,1}\cdot n^{-1} \gg n^{-1}$ which means the variance will dominate bias. As we shall see in Theorem~\ref{theorem: efficiency},
    $\|f\|^2_{s,\infty,1}\cdot n^{-1}$ is actually the optimal asymptotic variance.
	\end{remark}

In the next theorem, we prove a lower bound which together with (\ref{bound: efficiency}) implies the asymptotic efficiency of the proposed estimator $g(\bar{\bx})$.
\begin{theorem}
	\label{theorem: efficiency}
	Under model (\ref{model}), suppose that $f\in \calF^s(\RR^d)$ with $s>1$ and for some absolute constant $\tau>0$
	\begin{equation}
		\label{variance: lower_efficiency}
		\frac{n^{-1/2}\big\| f \big\|_{s,\infty,1}\|\bSigma\|^{1/2}_{op}}{\sigma_{f,\bepsilon}(\btheta_0)} \leq \tau .
	\end{equation}
	Then there exists an absolute constant $C_1>0$ such that for all $c>0$, the following bound holds
	\begin{equation}
		\label{bound: lecam_lower}
		\inf_T \sup\limits_{\btheta \in \calB(\btheta_0;cn^{-1/2})} \frac{\EE_{\btheta} (T(\bx) - f(\btheta))^2}{\sigma^2_{f,\bepsilon}(\btheta)}  \geq 1 - C_1 \tau^2 \big(  cn^{-(2\wedge s)/2+1/2} + c^{-2} \big) ,
	\end{equation}
	where $\calB(\btheta_0;cn^{-1/2}):= \{\btheta: \|\btheta -\btheta_0\|\leq cn^{-1/2} \}$. Especially,
	\begin{equation}
		\lim\limits_{c\rightarrow \infty} \lim_{n\rightarrow \infty} \inf_T \sup\limits_{\btheta \in \calB(\btheta_0;cn^{-1/2})}  \frac{\EE_{\btheta} (T(\bx) - f(\btheta))^2}{\sigma^2_{f,\bepsilon}(\btheta)} \geq 1 .
	\end{equation}
\end{theorem}

\begin{remark}
	Together with Theorem~\ref{theorem: normal}, Theorem~\ref{theorem: efficiency} shows that when $\sigma_{f,\bepsilon}(\btheta) \asymp \|\bSigma\|_{op}\|f\|_{s,\infty,1}/\sqrt{n}$, the asymptotic variance $\sigma^2_{f,\bepsilon}(\btheta)$ is optimal thus implies asymptotic efficiency of the proposed estimator $g(\bar{\bx})$. Furthermore, the variance $\|f\|^2_{s,\infty,1}/n$ matches the first term in (\ref{bound: MSE}) which shows that as long as $R^{-2s} = O(n^{-1})$, bound (\ref{bound: MSE}) is optimal given $s>1/(1-\alpha)$.
\end{remark}

\section{Estimation with Unknown Covariance Matrix}
\label{section: adaptive}

In this section, we discuss an adaptive estimation strategy to deal with the case when the covariance matrix $\bSigma$ is unknown. As we can see, there are two variables need to be decided without knowing $\bSigma$:
one is the true covariance matrix $\bSigma$ to be plugged into (\ref{estimator: f}) and the other is the truncation radius $R$ which can depend on $\br(\bSigma)$ when $\|f\|_{s,\infty,1} = O(1)$. Note that when $\|f\|_{s,\infty,1}$ depends on
the dimension, say $\|f\|_{s,\infty,1} = O(d^{s/2})$, one can simply choose $R = \sqrt{n}$ which does not depend on $d$.
Clearly, both parameters can be achieved with a fairly good estimator of $\bSigma$.
In the following, we provide a data driven method to estimate $\bSigma$ under model (\ref{model}), where multiple noisy observations are available.

Recall that $\bx_j = \btheta+ \bepsilon_j$, $j=1,...,n$ and $\bepsilon_j\sim \calN(\mathbf{0},\bSigma)$. To estimate $\bSigma$, we consider
\begin{equation}
	\bbeta_j = \sqrt{\frac{j-1}{j}}(\bx_j - \bar{\bx}_{j-1}),\hspace{0.1in}~\bar{\bx}_{j-1} = \frac{1}{j-1} \sum_{i=1}^{j-1} \bx_i,\hspace{0.1in}~j=2,...,n.
\end{equation}
 In this case, it is easy to check that $\widetilde{\bbeta}_{j} = \bbeta_{j+1}$, $j=1,...,n-1$ are i.i.d. copies of a centered Gaussian random vector $\bbeta\sim \calN(\mathbf{0},\bSigma)$. We denote by $\hat{\bSigma}:=(n-1)^{-1}\sum_{j=1}^{n-1}\widetilde{\bbeta}_j\widetilde{\bbeta}_j^T$ the sample
 covariance matrix of $\bbeta$, and we use $\hat{\bSigma} $ as an estimator of $\bSigma$. We denote by $\hat{\Omega} \subset \{\bzeta: \|\bzeta\|\leq \hat{R} \}$, then we define the following adaptive estimator
 \begin{equation}
 \label{estimator: adaptive}
 	\hat{g}(\bar{\bx}):=   \frac{1}{(2\pi)^{d/2}}   \int_{\hat{\Omega}}  \mathcal{F} f(\bzeta)e^{\langle \hat{\bSigma} \bzeta, \bzeta \rangle/2n} e^{i\bzeta \cdot \bar{\bx}} d\bzeta .
 \end{equation}
where $\hat{R}:= \sqrt{n/\br(\hat{\bSigma})}$ when $\big\| f \big\|_{s,\infty,1}=O(1)$ and $\hat{R} := \sqrt{n}$ when $\big\| f \big\|_{s,\infty,1} = O(d^{s/2})$.

In the following theorem, we show that the difference between $\hat{g}(\bar{\bx})$ and $g(\bar{\bx})$ is asymptotically negligible with high probability.

\begin{theorem}
\label{theorem: adaptive}
	Under model (\ref{model}), let $f\in \calF^s(\RR^d)$ with $s>1$ and $\hat{g}(\bar{\bx})$ and $g(\bar{\bx})$ be defined as in (\ref{estimator: adaptive}) and (\ref{estimator: f}) respectively. Then for any $t>1$ with probability at least $1-e^{-t}$ and some numerical constant  $\widetilde{C}$
	\begin{equation}
	\big| \hat{g}(\bar{\bx}) - g(\bar{\bx})\big| \leq \widetilde{C} \Big\{ \frac{\big\| f \big\|_{s,\infty,1}}{n^{(2\wedge s)/2}} \cdot \Big( \sqrt{\frac{\br(\bSigma)}{n}} \bigvee \sqrt{\frac{t}{n}} \Big)  + \big\| f \big\|_{s,\infty,1} \cdot R^{-s} \Big\}.
	\end{equation}
\end{theorem}

\begin{remark}
Theorem~\ref{theorem: adaptive} shows that $\big| \hat{g}(\bar{\bx}) - g(\bar{\bx})\big|$ is of smaller order than bound (\ref{bound: MSE}) on MSE achieved by $g(\bar{\bx})$ with high probability. Especially, when $\br(\bSigma) = O(n^{\alpha})$ with $\alpha\in(0,1)$ and $s>1/(1-\alpha)$, all the results we show for $g(\bar{\bx})$ still holds for $\hat{g}(\bar{\bx})$. Indeed, our simulation results in Section~\ref{section: simulation} show that both estimators achieve similar performance which is much
better than the plug-in estimator.
\end{remark}

\section{Numerical Simulation}
\label{section: simulation}

In this section, we conduct simulation study to test the performance of our estimator under standard Gaussian shift model where $\bepsilon \sim \mathcal{N}(\mathbf{0}, \bI_d)$.
We denote the estimator defined in (\ref{estimator: f}) by TF-estimator, and the estimator defined in (\ref{estimator: adaptive}) by ADP-estimator (adaptive estimator).
We test our estimators on the following type of multivariate functions:
$f(\btheta):=\beta*\prod_{j=1}^d h(\theta_j) $, where the normalizing factor $\beta$ is used to make $f(\btheta)$ be a constant for different values of $d$.

We choose $h$ with two different smoothness properties and compare the bias, variance, and MSE of TF-estimator, adaptive estimator with the plug-in estimator when
$\alpha$ ranges from $0.4$ to $0.85$.
The unknown parameters $\btheta\in \RR^d$ are randomly generated that yield a uniform distribution over $[0.4, 0.6]^d$ for different dimension parameter $d$.
We set the sample size $n=10000$.

We use the MATLAB built-in function $\mathbf{fft}()$ and $\mathbf{ifft}()$ to compute the Fourier Transformation and the Inverse Fourier Transformation appeared in the analysis. When we implement TF-estimator,
the truncation in the frequency domain was done uniformly for each coordinate for simplicity. Thus, the support of $\mathcal{F}f^N$ after truncation is contained in a hyper-cube instead of a d-ball. Note that the built-in function $\mathbf{fft}()$ and $\mathbf{ifft}()$ are implementations of discrete Fourier transform (DFT). Those discrepancies between implementations and our theoretical results make the cutoff range drift a little bit from our suggested choice of $R$. The cutoff range for each coordinate we use is $[64, 100]$. We observed that typically larger $\alpha$ and higher dimension $d$ needs smaller cutoff to achieve better performance, which is consistent with our prediction specified in $R$.

\subsection{Bias reduction}

\begin{figure}[!htb]
	\centering
	\begin{subfigure}[b]{0.48\textwidth}
		\includegraphics[width=\textwidth]{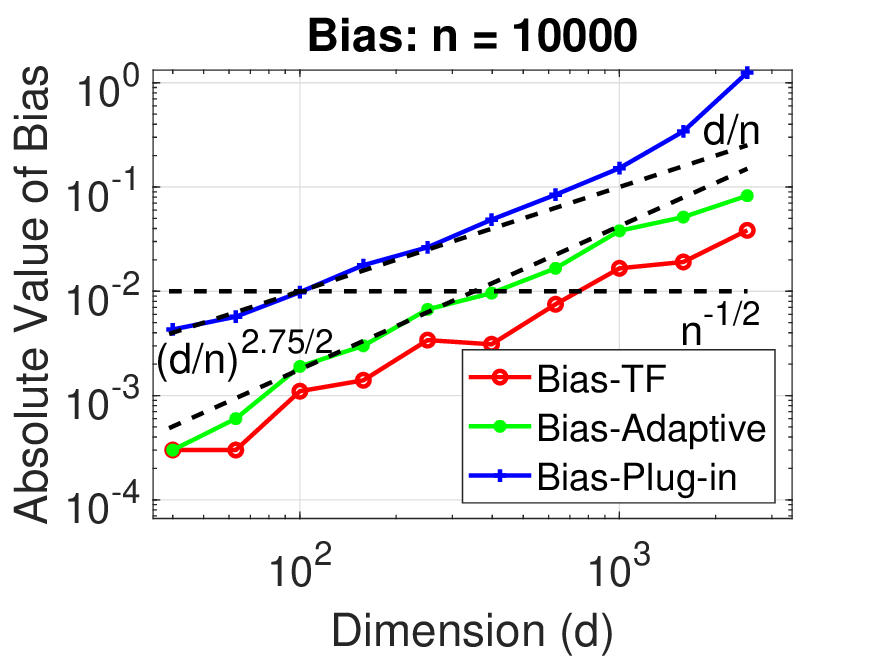}
		\caption{Bias Comparison for $h(x) = (2x)^{2.75}$}
		\label{fig: bias-s2}
	\end{subfigure}
	~ 
	\begin{subfigure}[b]{0.48\textwidth}
		\includegraphics[width=\textwidth]{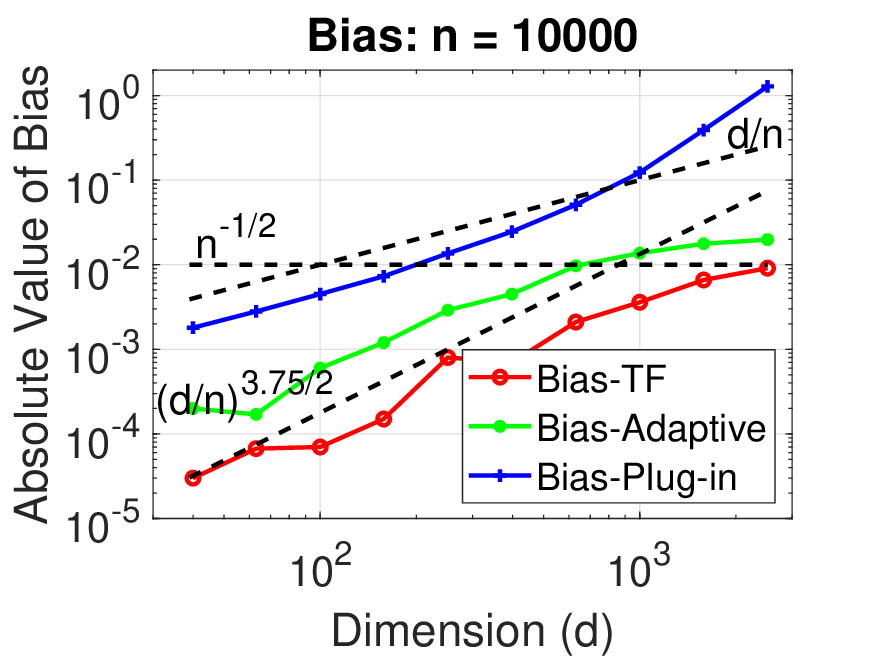}
		\caption{Bias Comparison for $h(x) = (2x)^{3.75}$}
		\label{fig: bias-s3}
	\end{subfigure}
	\caption{Bias Comparison}
\end{figure}

We choose two different base functions $h(x)$ with $x\in[0,1]$ to test the performance. One is $h_1(x) = (2x)^{2.75}$ and the other is $h_2(x) = (2x)^{3.75}$. The scalar factor is used to avoid overflow of the function values when the dimension $d$ is large. Especially, the underlying function values for both cases are normalized to a constant for this case in order to force the function magnitude to be bounded.
One should notice that $h_1(x) = (2x)^{2.75}$ belongs to the H\"{o}lder class with smoothness at most $s=2.75$ while $h_2(x) = (2x)^{3.75}$ belongs to the one with $s=3.75$.
In other words, $h_2$ has a higher smoothness condition than $h_1$.
The data for bias comparison for both cases are listed in Table~\ref{tb:adp_s2_bias} and Table~\ref{tb:adp_s3_bias} respectively in the appendix, and they are plotted in Figure~\ref{fig: bias-s2} and Figure~\ref{fig: bias-s3}.
The metric we use is $|\EE_{\btheta}g(\bar{\bx}) -f(\btheta)|$, where $\EE_{\btheta}g(\bar{\bx})$ is simulated by sample mean of 20000 independent trials.

As we can see, for both cases, the bias reduction phenomena are very obvious. The dash lines which plot $(d/n)^{2.75/2}$ and $(d/n)^{3.75/2}$ are supposed to be of the same order as the upper bounds on the bias of our estimators as proved in Theorem~\ref{theorem: bias}. The simulation results align with the bounds quite well. Sometimes, the actual bias can pass the line, we think these discrepancies may be due to the constant factors appeared in the bounds and the implementation issue we mentioned above.

Another phenomenon we are interested in is the threshold on smoothness. The magnitude of $f$ are intentionally adjusted for both cases such that the bias of the plug-in estimator will exceed the dash line $n^{-1/2}$ around $\alpha =0.5$. When we continue to increase $\alpha$ beyond $0.5$, the bias of TF-estimators and adaptive estimators still stays below the dash line $n^{-1/2}$ for sometime while the bias of the plug-in estimators exceeds way above this level. However, the bias of both adaptive estimators start to pass the line as $\alpha$ passes $0.65$ for the case $h(x) = (2x)^{2.75}$ and between $0.70$ and $0.75$ for the case $h(x) = (2x)^{3.75}$. The threshold on smoothness appeared in this article and~\cite{koltchinskii2018efficient} suggests that the bias are expected to be greater than $O(n^{-1/2})$ when $s>1/(1-\alpha)$. For both cases, the suggested passing point should be around $\alpha = 0.64$ for the case $h(x) = (2x)^{2.75}$ and the point should be around $\alpha = 0.73$ for the case $h(x) = (2x)^{3.75}$.

\subsection{MSE comparison and minimax lower bound}

\begin{figure}[!htb]
	\centering
	\begin{subfigure}[b]{0.48\textwidth}
		\includegraphics[width=\textwidth]{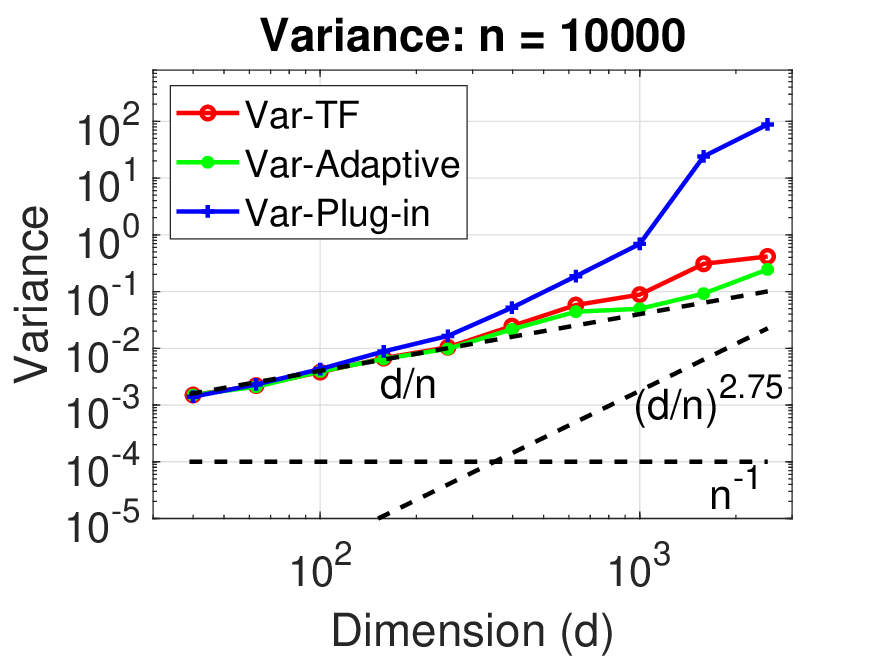}
		\caption{Variance Comparison for $h(x) = (2x)^{2.75}$}
		\label{fig: var-s2}
	\end{subfigure}
	~ 
	\begin{subfigure}[b]{0.48\textwidth}
		\includegraphics[width=\textwidth]{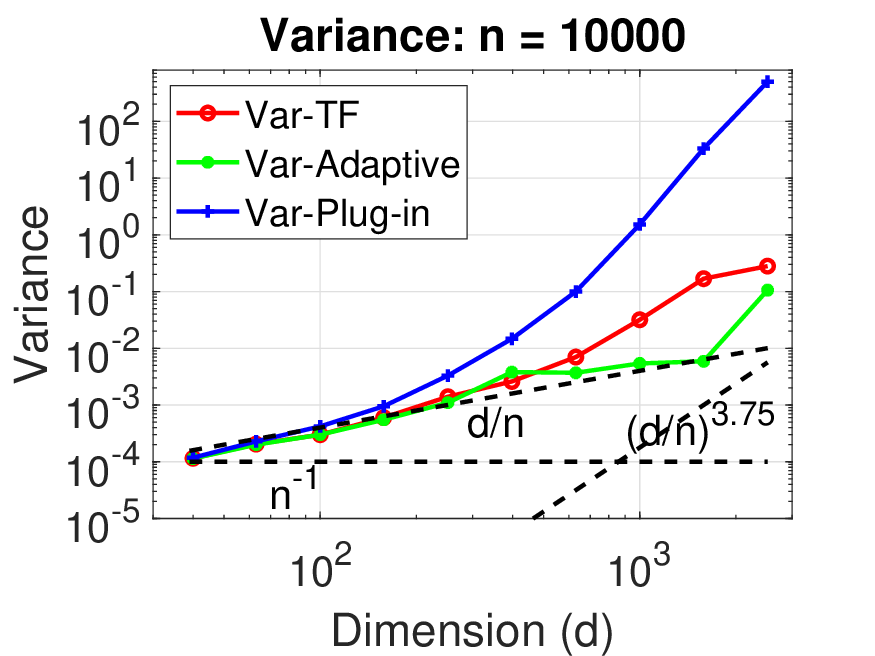}
		\caption{Variance Comparison for $h(x)=(2x)^{3.75}$}
		\label{fig: var-s3}
	\end{subfigure}
	\caption{Variance Comparison}
\end{figure}

\begin{figure}[!htb]
	\centering
	\begin{subfigure}[b]{0.48\textwidth}
		\includegraphics[width=\textwidth]{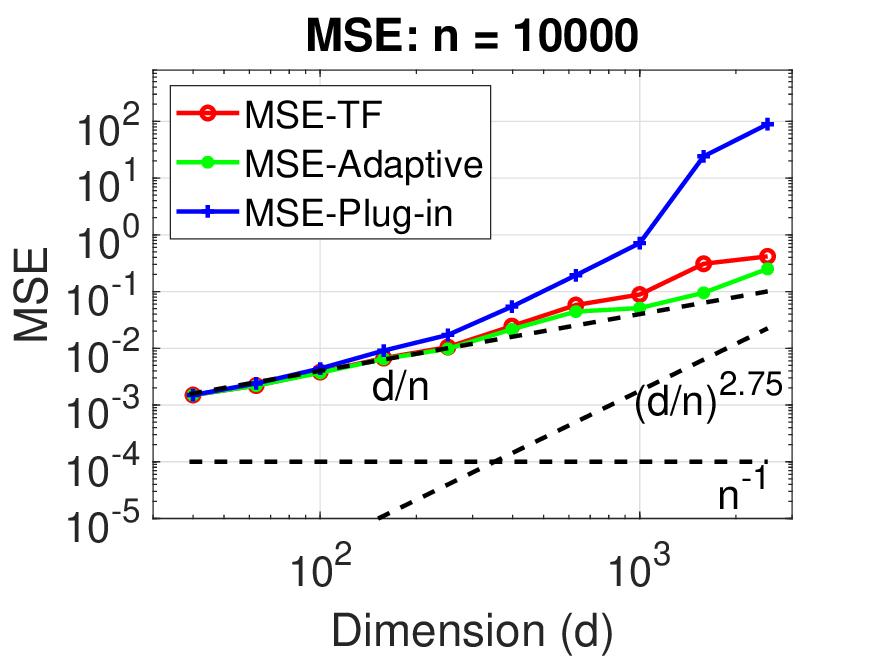}
		\caption{MSE Comparison for $h(x) = (2x)^{2.75}$}
		\label{fig: mse-s2}
	\end{subfigure}
	~ 
	\begin{subfigure}[b]{0.48\textwidth}
		\includegraphics[width=\textwidth]{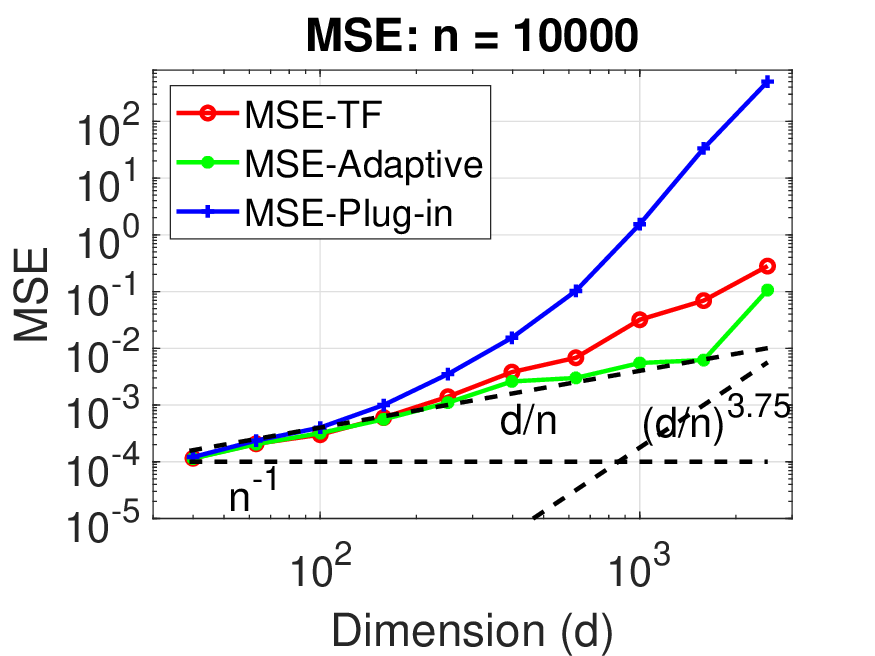}
		\caption{MSE Comparison for $h(x)=(2x)^{3.75}$}
		\label{fig: mse-s3}
	\end{subfigure}
	\caption{MSE Comparison}
\end{figure}

We compare the variance and MSE for both cases in this section.
The variance data are listed in Table~\ref{tb:adp_s2_var} and Table~\ref{tb:adp_s3_var} in appendix which are plotted in Figure~\ref{fig: var-s2} and Figure~\ref{fig: var-s3}. As we can see, the variance are almost the same for TF-estimator and adaptive estimator when $\alpha \leq 0.65$. However, when we keep increasing $d$ we can see a clear variance reduction compared with Plug-in estimator when $\alpha$ exceeds $0.5$.
Given that TF-estimator and adaptive estimator achieve better bias reduction, these show their superiority over Plug-in estimators.
Especially, the dash line $d/n$ is approximately of the same order as $\|\nabla f(\btheta)\|^2/n$ which is the lower bound according to well-known Cram\'{e}r-Rao bound. As we can see, when $\alpha$ is small, the variance of the proposed estimator aligns well with
this lower bound. However, when $\alpha$ is large, we can observe a clear divergence between these two. This validates Theorem~\ref{theorem: efficiency} that efficiency can not be achieved when the dimension is too large.

The metric for MSE we use is $\EE_{\btheta}(g(\bar{\bx}) -f(\btheta))^2$ which is simulated by averaging the square error of 20000 independent trials.
The MSE data are listed in Table~\ref{tb:adp_s2_mse} and Table~\ref{tb:adp_s3_mse} in appendix and are plotted in Figure~\ref{fig: mse-s2} and Figure~\ref{fig: mse-s3}.
As we can see, the improvements on reduction of MSE become more obvious as the dimension grow larger for both cases. We also plot $ n^{-1}$ and $(d/n)^s$ as dash lines, which are supposed to be of the same order as the minimax lower bounds on MSE when
$\|f\|_{s,\infty,1}=O(1)$ as shown in (\ref{bound: lower_2}) and (\ref{bound: lower_1}) and $d/n$ is supposed to be minimax lower bound on MSE when $\|f\|_{s,\infty,1}\gg 1$. In Figure~\ref{fig: mse-s2} and Figure~\ref{fig: mse-s3}, we can see that both MSE curves' trend align well with the bounds. Meanwhile, we can see that when $\alpha$ exceeds $0.5$, the reduction in MSE becomes more obvious for both cases. Especially, the reduction with $h(x)=(2x)^{3.75}$ with more smoothness is more obvious than with $h(x)=(2x)^{2.75}$.

\subsection{Applications in Building Confidence Intervals}

\begin{table}[!htb]
	\centering
	\begin{center}
		\begin{tabular}{|c|c|c|c|}
			\hline\hline
			$\alpha$-value  & True parameter & Plug-in &  Adaptive \\ \hline
			0.40             &        0.8578    &  [0.8880, 0.8970]       &   \bf{[0.8517, 0.8606]}  \\ \hline
			0.45             &        0.7218    &  [0.7647, 0.7746]       &   \bf{[0.7226, 0.7324]}        \\ \hline
			0.50             &        0.7795    &  [0.8526, 0.8668]       &   \bf{[0.7768, 0.7904] } \\ \hline
			0.55             &        0.8917    &  [1.0187, 1.0407]       &   \bf{[0.8901, 0.9105] } \\ \hline
			0.60             &        0.3415    &  [0.4222, 0.4347]       &   \bf{[0.3398, 0.3505] }  \\ \hline
			0.65             &        0.2946    &  [0.4300, 0.4493]       &   \bf{[0.2942, 0.3086]}  \\ \hline
			0.70             &        1.2035    &  [2.1257, 2.2776]       &   [1.2844, 1.3884]  \\ \hline
			0.75             &        0.4020    &  [0.9884, 1.1010]       &   [0.4198, 0.4744]    \\ \hline
			0.80             &        2.1655    &  [9.1964, 12.569]       &   [2.6274, 3.8089]   \\ \hline
			0.85             &        0.2988    &  [2.2443, 6.4455]       &   [0.2367, 0.9197]   \\ \hline
			\hline
		\end{tabular}
		\caption{95\% Confidence interval to estimate $f(\btheta)$ with $\beta = 2.75$}
		\label{tb:ci-275}
	\end{center}
\end{table}

As we have mentioned in Section~\ref{section: normal}, an important application of normal approximation in practice is to use the estimates to build confidence intervals of the true parameter.
To compare the quality of confidence intervals built by the estimator, we collect the estimates of a fixed underlying parameter $f(\btheta)$ from 20000 independent runs. We use the MATLAB built-in function $\mathbf{histfit()}$ to draw the histograms and use $\mathbf{fitdist()}$ to fit the histograms into a normal distribution.
We show the 95\% confidence intervals for estimation of $f(\btheta)$ from the fitted normal models in Table~\ref{tb:ci-275} with $\beta = 2.75$.
As we can see, confidence intervals based on the adaptive estimator (\ref{estimator: adaptive}) are much better than the those based on the plug-in approach at all levels of dimension.
In fact, for small $\alpha$ values, say $\alpha \leq 0.65$, the true parameter always falls into the confidence interval we built while falls outside the ones built by the plug-in estimator.
However, when $\alpha > 0.65$, the confidence intervals built based on both approaches become quite unreliable. This phenomenon aligns well with our theory in Theorem~\ref{theorem: normal} which suggests that when dimension is large enough such that the smoothness index $s$ is below the threshold $1/(1-\alpha)$, then the bias can be large and the asymptotic normality result may fail to hold.

\section{Conclusion and Discussion}

In this article, we studied the estimation of $f(\btheta)$ with an unknown parameter $\btheta\in \RR^d$ and a given $f$ under Gaussian shift model when the intrinsic dimension of the parameter can grow with the sample size.
We proposed a new estimator which can be shown both analytically and experimentally to achieve much better bias reduction and variance reduction than $f(\bar{\bx})$. Asymptotic normality and efficiency were proved once the
smoothness parameter stays above some threshold related to dimensionality. Such threshold was initially discovered by~\cite{koltchinskii2018efficient} and turns out to be sharp for a H\"{o}lder type smooth class with
$\|f\|_{C^s}\leq 1$. However, as we discussed in Section \ref{section: prelim}, the norm $\|f\|_{s,\infty,1}$ introduced in this paper can easily depend on $d$ which can be much larger than 1, we are particularly interested
in whether with large $\|f\|_{s,\infty,1}$ such threshold is still sharp or not. To answer this question, one needs a proper minimax lower bound characterized by $\|f\|_{s,\infty,1}$ and $s$. Another aspect of our particular interest is that whether current results under Gaussian shift model can be generalized to general distributions such as heavy-tailed ones. The estimator proposed in this article may not be directly applied in such cases since it heavily relies on the gaussian assumption. In order to solve this problem, one may need to develop some new estimator. Generalization of the results in~\cite{koltchinskii2018efficient} also can be hard as some new concentration inequality needs to be developed as indicated in~\cite{koltchinskii2019estimation}. We leave those questions as future directions.

\section*{Acknowledgement}

We sincerely thank the insightful discussion with Professor Cun-Hui Zhang  during his visit at Baidu Research and thank Professor Vladimir Koltchinskii for several insightful comments on this work.

\newpage

\section{Proofs}

\subsection{Proof of Theorem~\ref{theorem: bias}}
\begin{proof}
	\label{proof: theorem_bias}
	We decompose $f$ as $f:= f^N + \tilde{f}^N$, where
	\begin{equation*}
		\tilde{f}^N(\btheta) := (2\pi)^{-d/2} \int_{\RR^d\backslash\Omega} \calF f(\bzeta)\cdot e^{i\bzeta^T \btheta}  d\bzeta
	\end{equation*}
	Recall that from (\ref{conv2}), we have
	\begin{equation}
	\begin{aligned}
	\EE_{\theta} g(\bar{\bx}) = g*p^o(\btheta) &=  \frac{1}{(2\pi)^{d/2}}  \int_{\mathbb{R}^d} \Big(\int_{\mathbb{R}^d}  \mathcal{F} f^N(\bzeta)e^{\langle \bSigma \bzeta, \bzeta \rangle/2} e^{i\bzeta \cdot (\btheta-\bx)} d\bzeta \Big) p^o(\bx) d\bx \\
	& = \int_{\mathbb{R}^d}  \mathcal{F} f^N(\bzeta)e^{\langle \bSigma \bzeta, \bzeta \rangle/2} e^{i\bzeta \cdot \btheta} \mathcal{F} p^o(\bzeta) d\bzeta \\
	& = \frac{1}{(2\pi)^{d/2}} \int_{\mathbb{R}^d}  \mathcal{F} f^N(\bzeta) e^{i\bzeta \cdot \btheta}  d\bzeta \\
	& =  \mathcal{F}^{-1} \mathcal{F} f^N = f^N(\btheta),
	\end{aligned}
	\end{equation}
	where for the third line, we used the fact that $\mathcal{F} p^o(\bzeta) = (2\pi)^{-d/2} e^{-\langle \bSigma \bzeta, \bzeta \rangle/2}.$
	Therefore, $\big| \EE_{\theta} g(\bar{\bx}) - f(\btheta)\big| = |f^N(\btheta) - f(\btheta)| = |\tilde{f}^N(\btheta)| $. Meanwhile, for $\tilde{f}^N(\btheta)$ we denote by
	$\Omega^c:= \RR^d \backslash \Omega$ the complementary of domain $\Omega$ of $\calF f^N$
	\begin{equation}
	\label{norm: remainder}
	\begin{aligned}
	|\tilde{f}^N(\btheta)| & \leq \Big| \frac{1}{(2\pi)^{d/2}} \int_{\Omega^c}  \mathcal{F} f(\bzeta) e^{i\bzeta \cdot \btheta}  d\bzeta  \Big| \\
	& \leq   \frac{1}{(2\pi)^{d/2}} \int_{\Omega^c}  \Big| \mathcal{F} f(\bzeta) \Big| \cdot \|\bzeta \|^s \cdot R^{-s} d\bzeta  \\
	& \leq   \frac{1}{(2\pi)^{d/2}} \int_{\Omega^c}  \Big| \mathcal{F} f(\bzeta) \Big|  d\bzeta \leq \big\| f \big\|_{s,\infty,1} \cdot R^{-s},
	\end{aligned}
	\end{equation}
	where $R$ denotes the radius of $\Omega$.
	Especially, by taking $R = \sqrt{n/\br(\bSigma)}$
	\begin{equation}
	\big|\tilde{f}^N(\btheta) \big|  \leq \big\| f \big\|_{s,\infty,1} \cdot \Big( \frac{\br(\bSigma)}{n} \Big) ^{s/2} .
	\end{equation}
\end{proof}

\subsection{Proof of Theorem~\ref{theorem: normal}}
\begin{proof}[Proof of Theorem~\ref{theorem: normal}]
	We denote by $g_{\bzeta,n}(\bepsilon) := 1+ i\bzeta^T \bepsilon/n$, and consider the following decomposition of $g(\bar{\bx}) - f(\btheta)$:
	\begin{equation}
		\begin{aligned}
			g(\bar{\bx}) - f(\btheta) = \calA_1 + \calA_2 + \calA_3,
		\end{aligned}
	\end{equation}
	where we denote by
	\begin{align}
		\calA_1 &:= g(\bar{\bx}) - \frac{1}{(2\pi)^{d/2}}   \int_{\Omega}  \mathcal{F} f^N(\bzeta) \prod_{j=1}^n g_{\bzeta,n} (\bepsilon_j) e^{i\bzeta^T\btheta}d\bzeta \\
		\calA_2 &:=  \frac{1}{(2\pi)^{d/2}}   \int_{\Omega}  \mathcal{F} f^N(\bzeta) \prod_{j=1}^n g_{\bzeta,n} (\bepsilon_j) e^{i\bzeta^T\btheta}d\bzeta  - f^N(\btheta) \\
		\calA_3 &:=  \tilde{f}^N(\btheta).
	\end{align}
	We deal with $\calA_1$, $\calA_2$, and $\calA_3$ respectively.
	
	Firstly, we show that $\calA_1/\sigma_{f,\bepsilon}(\btheta)$ converges to 0 in probability. Recall that by the definition of $g(\bar{\bx})$ we have
	\begin{equation}
		\begin{aligned}
			\calA_1 :=  \frac{1}{(2\pi)^{d/2}} \int_{\Omega}  \mathcal{F} f^N(\bzeta)e^{i\bzeta^T\btheta}
			\Big(e^{\langle\bSigma\bzeta,\bzeta\rangle/2n}\prod_{j=1}^n e^{i\bzeta^T\bepsilon_j/n} - \prod_{j=1}^n g_{\bzeta,n} (\bepsilon_j)\Big) d\bzeta
		\end{aligned}
	\end{equation}
	Therefore,
	\begin{equation}
		\begin{aligned}
			\EE \big| \calA_1 \big|  &= \EE \Big| \frac{1}{(2\pi)^{d/2}} \int_{\Omega}  \mathcal{F} f^N(\bzeta) e^{i\bzeta^T\btheta}
			\EE \Big[ e^{\langle\bSigma\bzeta,\bzeta\rangle/2n}\prod_{j=1}^n e^{i\bzeta^T\bepsilon_j/n} - \prod_{j=1}^n g_{\bzeta,n} (\bepsilon_j)\Big] d\bzeta \Big| \\
			& \leq\frac{1}{(2\pi)^{d/2}} \int_{\Omega}  \Big| \mathcal{F} f^N(\bzeta)  \Big|
			\cdot \EE \Big| e^{\langle\bSigma\bzeta,\bzeta\rangle/2n}\prod_{j=1}^n e^{i\bzeta^T\bepsilon_j/n} - \prod_{j=1}^n g_{\bzeta,n} (\bepsilon_j)\Big| d\bzeta
		\end{aligned}
	\end{equation}
	The next step is to bound
	$\EE \big| e^{\langle\bSigma\bzeta,\bzeta\rangle/2n}\prod_{j=1}^n e^{i\bzeta^T\bepsilon_j/n} - \prod_{j=1}^n g_{\bzeta,n} (\bepsilon_j)\big|$.
	By a standard swapping argument,
	\begin{equation}
		\label{A1}
		\begin{aligned}
			& \EE \big| e^{\langle\bSigma\bzeta,\bzeta\rangle/2n}\prod_{j=1}^n e^{i\bzeta^T\bepsilon_j/n} - \prod_{j=1}^n g_{\bzeta,n} (\bepsilon_j)\big|\\
			& \leq \sum_{k=1}^n \EE\Big|\prod_{j=1}^{k-1}e^{\langle\bSigma\bzeta,\bzeta\rangle/2n^2} e^{i\bzeta^T\bepsilon_j/n}
			\Big(e^{\langle\bSigma\bzeta,\bzeta\rangle/2n^2} e^{i\bzeta^T\bepsilon_k/n} -g_{\bzeta,n} (\bepsilon_k) \Big)\prod_{i=k+1}^{n} g_{\bzeta,n} (\bepsilon_i)   \Big|\\
			& \leq  \sum_{k=1}^n e^{(k-1)_+\langle\bSigma\bzeta,\bzeta\rangle/2n^2} \cdot \EE \Big|e^{\langle\bSigma\bzeta,\bzeta\rangle/2n^2} e^{i\bzeta^T\bepsilon_k/n} -g_{\bzeta,n} (\bepsilon_k) \Big|
			\cdot \Big( \EE\big| g_{\bzeta,n} (\bepsilon_k)\big| \Big)^{(n-k)_+}.
		\end{aligned}
	\end{equation}
	Recall that for $e^{\langle\bSigma\bzeta,\bzeta\rangle/2n^2}$ we have
	\begin{equation}
		\label{comp1}
		e^{\langle\bSigma\bzeta,\bzeta\rangle/2n^2} \leq e^{\sigma^2\| \bzeta\|^2 /2n^2} \leq e^{\sigma^2R^2 /2n}.
	\end{equation}
	For $\EE \big|e^{\langle\bSigma\bzeta,\bzeta\rangle/2n^2} e^{i\bzeta^T\bepsilon_k/n} -g_{\bzeta,n} (\bepsilon_k) \big|$,
	\begin{equation}
		\label{comp2}
		\begin{aligned}
			& \EE \big|e^{\langle\bSigma\bzeta,\bzeta\rangle/2n^2} e^{i\bzeta^T\bepsilon_k/n} -g_{\bzeta,n} (\bepsilon_k) \big|  \\
			& \leq \EE \big|e^{\langle\bSigma\bzeta,\bzeta\rangle/2n^2}  - 1 \big| + \EE \big| e^{i\bzeta^T\bepsilon_k/n}  -  g_{\bzeta,n} (\bepsilon_k) \big| \\
			& \leq e\sigma^2\| \bzeta\|^2 /2n^2 +  \sigma^2\| \bzeta\|^2 /2n^2 .
		\end{aligned}
	\end{equation}
	For $\EE\big| g_{\bzeta,n} (\bepsilon_k)\big|$, we have
	\begin{equation}
		\label{comp3}
		\EE\big| g_{\bzeta,n} (\bepsilon_k)\big| \leq \sqrt{\EE \big| 1 + i\bzeta^T\bepsilon_k/n\big|^2  } \leq  \sqrt{1+\sigma^2\| \bzeta\|^2 /n^2 }  \leq \sqrt{(1+\sigma^2R^2/n^2)} .
	\end{equation}
	Combine (\ref{comp1}), (\ref{comp2}) and (\ref{comp3}), we get
	\begin{equation}
		\begin{aligned}
			&\EE \big| e^{\langle\bSigma\bzeta,\bzeta\rangle/2n}\prod_{j=1}^n e^{i\bzeta^T\bepsilon_j/n} - \prod_{j=1}^n g_{\bzeta,n} (\bepsilon_j)\big|
			& \leq  \sum_{k=1}^n 2e^2\sigma^2 \| \bzeta\|^2  n^{-2}  \leq 2e^2\sigma^2 \| \bzeta\|^2  n^{-1} .
		\end{aligned}
	\end{equation}
	Plug this into (\ref{A1}), by the definition of $\|f\|_{s,\infty,1}$ and $R\leq \sqrt{n}$, we obtain
	\begin{equation}
		\begin{aligned}
			\EE \big| \calA_1 \big| & \leq \frac{1}{(2\pi)^{d/2}} \int_{\Omega}  \Big| \mathcal{F} f^N(\bzeta)  \Big|
			\cdot 2e^2\sigma^2 \| \bzeta\|^2  n^{-1}  d\bzeta \\
			& \leq \frac{1}{(2\pi)^{d/2}} \int_{\Omega}  \Big| \mathcal{F} f^N(\bzeta)  \Big| \cdot 2e^2\sigma^2 \cdot \frac{( 1 \vee \| \bzeta\|^s)}{n^{(2 \wedge s)/2}}  d\bzeta \\
			& \leq 2e^2\sigma^2 \big\|f \big\|_{s,\infty,1} \cdot n^{-(2 \wedge s)/2}.
		\end{aligned}
	\end{equation}
	As a result, we have
	\begin{equation}
		\frac{\EE \big| \calA_1 \big|}{\sigma_{f,\bepsilon}(\btheta)} \leq \frac{2e^2\sigma^2 \big\|f \big\|_{s,\infty,1}  \cdot n^{-(2 \wedge s)/2} }{n^{-1/2} \sqrt{\big\langle \bSigma \nabla f(\btheta) , \nabla f(\btheta)\big\rangle }}
		\lesssim \tau n^{-(2 \wedge s)/2 + 1/2}.
	\end{equation}
	Under the condition of Theorem~\ref{theorem: normal}, we have
	\begin{equation}
		\frac{\EE \big| \calA_1 \big|}{\sigma_{f,\bepsilon}(\btheta)} \rightarrow 0,~~{\rm as}~~n\rightarrow \infty,
	\end{equation}
	which implies that
	\begin{equation}
		\label{A_1}
		\frac{\calA_1}{\sigma_{f,\bepsilon}(\btheta)} \xrightarrow{p} 0 .
	\end{equation}
	
	The next major step is to show that $\calA_2/\sigma_{f,\bepsilon}(\btheta) \Rightarrow \calN(0,1)$. We can rewrite $\calA_2$ as
	\begin{equation}
		\begin{aligned}
			\calA_2 := \frac{1}{(2\pi)^{d/2}}   \int_{\Omega}  \mathcal{F} f^N(\bzeta) \Big( \prod_{j=1}^n g_{\bzeta,n} (\bepsilon_j) - 1\Big) e^{i\bzeta^T\btheta}d\bzeta .
		\end{aligned}
	\end{equation}
	Further,
	\begin{equation}
		\begin{aligned}
			\prod_{j=1}^n g_{\bzeta,n} (\bepsilon_j)  &= \prod_{k=1}^n  \Big( 1 + i\bzeta^T\bepsilon_k/n \Big) \\
			& = \Big(  1 + S_1(\bzeta) + S_2(\bzeta) +\cdots +S_n(\bzeta)\Big) ,
		\end{aligned}
	\end{equation}
	where $S_k(\bzeta) = \sum_{1\leq j_1\neq\cdots \neq j_k\leq n} \prod _{i=1}^k (i\bzeta^T\bepsilon_{j_i}/n )$. Due to $\EE\bepsilon_k = 0$, $k=1,\dots,n$ and
	independence between $\bepsilon_k$'s, we have $\EE S_k(\bzeta) = 0$. On the other hand, variance of each $S_k(\bzeta)$
	\begin{equation}
		\EE \big|  S_k(\bzeta)  \big|^2 = {n \choose k} \Big( \EE \big|\bzeta^T\bepsilon_{1}/n \big|^2\Big)^k \leq {n \choose k} \Big( \sigma^2 \big\|\bzeta\big\|^2 /n^2\Big)^k \leq \frac{\sigma^{2k}  \big\|\bzeta\big\|^{2k} }{k!n^{k}}.
	\end{equation}
	Note that $S_k$ are completely degenerate U-statistics of order $k$.
	As a consequence,
	\begin{equation}
		\calA_2 := \frac{1}{(2\pi)^{d/2}}   \int_{\Omega}  \mathcal{F} f^N(\bzeta) \Big( S_1(\bzeta) + S_2(\bzeta) +\cdots +S_n(\bzeta)\Big) e^{i\bzeta^T\btheta}d\bzeta,
	\end{equation}
	and above analysis shows that for the case $k=1$
	\begin{equation}
		\frac{1}{(2\pi)^{d/2}}  \int_{\Omega}  \mathcal{F} f^N(\bzeta) S_1(\bzeta) e^{i\bzeta^T\btheta} d\bzeta = \frac{1}{(2\pi)^{d/2}}  \int_{\Omega}  \mathcal{F} f^N(\bzeta) \Big( \sum_{j=1}^n i\bzeta^T\bepsilon_k/n \Big)e^{i\bzeta^T\btheta} d\bzeta ,
	\end{equation}
	which is a sum of zero mean i.i.d. random variables, and the variance of each component is
	\begin{equation}
		\label{bound: variance_1}
		\begin{aligned}
			&\frac{1}{(2\pi)^{d}} \int_{\Omega} \int_{\Omega}   \mathcal{F} f^N(\bzeta_1)\mathcal{F} f^N(\bzeta_2) e^{i \bzeta_1^T\btheta} e^{-i \bzeta_2^T\btheta} \cdot \frac{\bzeta_1^T\bSigma \bzeta_2}{n^2}  d\bzeta_1d\bzeta_2 \\
			&=  \frac{1}{(2\pi)^{d}} \int_{\Omega} \int_{\Omega}   \mathcal{F} f^N(\bzeta_1)\mathcal{F} f^N(\bzeta_2) e^{i \bzeta_1^T\btheta} e^{-i \bzeta_2^T\btheta} \cdot n^{-2} \big\langle \bSigma i\bzeta_1, -i\bzeta_2\big\rangle d\bzeta_1d\bzeta_2 \\
			& = n^{-2} \big\langle \bSigma \nabla f^N(\btheta) , \nabla f^N(\btheta)\big\rangle,
		\end{aligned}
	\end{equation}
	where for the last equality, we used the following identity
	\begin{equation}
		\begin{aligned}
			\frac{1}{(2\pi)^{d/2}} \int_{\Omega} \mathcal{F} f^N(\bzeta) i\bzeta e^{i \bzeta^T\btheta}  d\bzeta & = \nabla_{\btheta} \Big[\frac{1}{(2\pi)^{d/2}} \int_{\Omega} \mathcal{F} f^N(\bzeta)  e^{i \bzeta^T\btheta}  d\bzeta \Big] \\
			& = \nabla_{\btheta} \calF^{-1} \calF f^N(\btheta) = \nabla_{\btheta} f^N(\btheta).
		\end{aligned}
	\end{equation}
	By standard Central Limit Theorem,
	\begin{equation}
		\label{s_1}
		\frac{\sqrt{n}\cdot \frac{1}{(2\pi)^{d/2}}  \int_{\Omega}  \mathcal{F} f^N(\bzeta) S_1(\bzeta)  d\bzeta}{ \sqrt{\big\langle \bSigma \nabla f^N(\btheta) , \nabla f^N(\btheta)\big\rangle}} \Rightarrow \calN(0,1).
	\end{equation}
	Meanwhile, to replace the variance $\sigma_{f^N,\bepsilon}(\btheta)$ by $\sigma_{f,\bepsilon}(\btheta)$, we have the following lemma.
	\begin{lemma}
		\label{lemma: variance}
		For a given $f \in \calF^s(\RR^d)$ with $d=n^{\alpha}$ and $s>1/(1-\alpha)$, and assume that $R\leq \sqrt{n}$. Then we have
		\begin{equation}
			\big| \sigma_{f^N,\bepsilon}(\btheta) - \sigma_{f,\bepsilon}(\btheta) \big| =  o\big( n^{-1/2} \big).
		\end{equation}
	\end{lemma}
	According to Lemma~\ref{lemma: variance}, we obtain
	\begin{equation}
		\begin{aligned}
			\frac{\frac{1}{(2\pi)^{d/2}}  \int_{\Omega}  \mathcal{F} f^N(\bzeta) S_1(\bzeta)  d\bzeta}{ \sigma_{f,\bepsilon}(\btheta)} & = \frac{ \frac{1}{(2\pi)^{d/2}}  \int_{\Omega}  \mathcal{F} f^N(\bzeta) S_1(\bzeta)  d\bzeta}{ \sigma_{f^N,\bepsilon}(\btheta)} \cdot \frac{\sigma_{f^N,\bepsilon}(\btheta)}{\sigma_{f,\bepsilon}(\btheta)} \\
			& =  \frac{\frac{1}{(2\pi)^{d/2}}  \int_{\Omega}  \mathcal{F} f^N(\bzeta) S_1(\bzeta)  d\bzeta}{ \sigma_{f^N,\bepsilon}(\btheta)} \cdot ( 1 \pm o(n^{-1/2}))
		\end{aligned}
	\end{equation}
	which implies that
	\begin{equation}
		\frac{ \frac{1}{(2\pi)^{d/2}}  \int_{\Omega}  \mathcal{F} f^N(\bzeta) S_1(\bzeta)  d\bzeta}{ \sigma_{f,\bepsilon}(\btheta)}  \Rightarrow \calN(0,1),~as~n\rightarrow \infty.
	\end{equation}
	As for the case $u\geq 2$, we have
	\begin{equation}
		\begin{aligned}
			& \EE \Big| \frac{1}{(2\pi)^{d/2}}  \int_{\Omega}  \mathcal{F} f^N(\bzeta) \sum_{u=2}^n S_u(\bzeta) e^{i\bzeta^T\btheta} d\bzeta  \Big|  \leq  \frac{1}{(2\pi)^{d/2}}   \int_{\Omega}  \Big| \mathcal{F} f^N(\bzeta) \Big| \EE \big| \sum_{u=2}^n S_u(\bzeta) \big|  d\bzeta \\
			& \leq  \frac{1}{(2\pi)^{d/2}}   \int_{\Omega}  \Big| \mathcal{F} f^N(\bzeta) \Big| \sqrt{\EE \Big| \sum_{u=2}^n S_u(\bzeta) \Big|^2 } d\bzeta
			= \frac{1}{(2\pi)^{d/2}}   \int_{\Omega}  \Big| \mathcal{F} f^N(\bzeta) \Big| \sqrt{ \sum_{u=2}^n \EE \big| S_u(\bzeta) \big|^2 } d\bzeta  \\
			&\leq \frac{1}{(2\pi)^{d/2}}   \int_{\Omega}  \Big| \mathcal{F} f^N(\bzeta) \Big| \cdot \sqrt{ \sum_{u=2}^n \frac{\tau_{n,2}^u(\bzeta)}{u!} }d\bzeta
			\leq \frac{1}{(2\pi)^{d/2}}   \int_{\Omega}  \Big| \mathcal{F} f^N(\bzeta) \Big| \cdot \tau_{n,2} (\bzeta) d\bzeta  \\
			&\lesssim  2\sigma^2 \big\|f \big\|_{s,\infty,1} \cdot n^{-(2 \wedge s)/2}.
		\end{aligned}
	\end{equation}
	The second line is due to Jensen's inequality; the third inequality is due to the bound on $\EE \big| S_k(\bzeta) \big|^2$ and $\tau_{n,2} \leq 1$, and the last line is due to the definition of $\big\| f \big\|_{\infty,s,1}$ .
	As a result, we obtain
	\begin{equation}
		\begin{aligned}
			\frac{\EE \Big| \frac{1}{(2\pi)^{d/2}}  \int_{\Omega}  \mathcal{F} f^N(\bzeta) \sum_{k=2}^n S_k(\bzeta) e^{i\bzeta^T\btheta} d\bzeta  \Big| }{n^{-1/2} \sqrt{\big\langle \bSigma \nabla f(\btheta) , \nabla f(\btheta)\big\rangle }}
			& \leq  \frac{2\sigma^2 \big\|f \big\|_{s,\infty,1} \cdot n^{-(2 \wedge s)/2} }{n^{-1/2}\sqrt{\big\langle \bSigma \nabla f(\btheta) , \nabla f(\btheta)\big\rangle }}
			&\lesssim \tau n^{-(2 \wedge s)/2 + 1/2}.
		\end{aligned}
	\end{equation}
	It implies that when $n\rightarrow \infty$,
	\begin{equation}
		\label{s_k}
		\frac{\EE \Big| \frac{1}{(2\pi)^{d/2}}  \int_{\Omega}  \mathcal{F} f(\bzeta) \sum_{k=2}^n S_k(\bzeta) e^{i\bzeta^T\btheta} d\bzeta  \Big| }{n^{-1/2} \sqrt{\big\langle \bSigma \nabla f(\btheta) , \nabla f(\btheta)\big\rangle }}  \rightarrow 0,
	\end{equation}
	which implies that it converges to 0 in probability.
	Combine (\ref{s_1}) and (\ref{s_k}), by Slutsky's theorem we show that
	\begin{equation}
		\frac{\sqrt{n} \cdot \calA_2}{ \sqrt{\big\langle \bSigma \nabla f(\btheta) , \nabla f(\btheta)\big\rangle }} \Rightarrow \calN(0,1),~~{\rm as}~~n\rightarrow \infty.
	\end{equation}
	
	As for $\calA_3$, as in the proof of Theorem~\ref{theorem: bias} we have
	\begin{equation}
		\label{A_3}
		\frac{\sqrt{n} \cdot \calA_3}{ \sqrt{\big\langle \bSigma \nabla f(\btheta) , \nabla f(\btheta)\big\rangle }} \leq \frac{ \big\| f \big\|_{s,\infty,1} \cdot \sqrt{n} \cdot R^{-s}}{ \sqrt{\big\langle \bSigma \nabla f(\btheta) , \nabla f(\btheta)\big\rangle }}  \leq  \tau \sqrt{n} \cdot R^{-s}
	\end{equation}
	which goes to zero as $n\rightarrow \infty$.
	
	Together with (\ref{A_1}), (\ref{s_1}) and (\ref{A_3}), we have showed that
	\begin{equation}
		\frac{g(\bar{\bx}) - f(\btheta)}{\sigma_{f,\bepsilon}(\btheta)}  \Rightarrow \calN(0,1),~~{\rm as}~~n\rightarrow \infty.
	\end{equation}
	
\end{proof}

\subsection{Proof of Theorem~\ref{theorem: minimax_lower_2}}

The method we use to prove Theorem~\ref{theorem: minimax_lower_2} is new and different from the proof of Theorem 2.2 in~\cite{koltchinskii2018efficient}
and it can be applied to the one sample situation under model (\ref{model}).
\begin{proof}
	We consider
	$$
	\bx = \btheta + \bxi,~\btheta\in \RR^d,~\bxi\sim \mathcal{N}(0; n^{-1}\cdot \bI_d)
	$$
	with unknown mean $\btheta$.
	Let $\bTheta :=\{ \btheta_0,...,\btheta_{M-1}\}$ be a set of $M = 2^d$ points such that $\|\btheta_i\| = 8\varepsilon$, $\|\btheta_i - \btheta_j\| \geq 2 \varepsilon$, $0\leq i,j\leq M-1$, $i\neq j$, where $\varepsilon \leq 1/8$.
	Let $\varphi$ be a function with compact support in $[0,1]$ and $\varphi(0) = a>0$ with $a$ being a constant $a\leq 1$. Based on $\varphi$, we define $\tilde{\varphi}: \RR^d \rightarrow \RR$
	such that $\tilde{\varphi}(\bt) := \varphi(\|\bt\|^2)$ and $\|\tilde{\varphi}\|_{s,\infty,1} \leq B$ (Note that such $\tilde{\varphi}$ exists and one can show $B \lesssim d^s$). For $i\in \{0,1,...,M-1\}$ and $\ell = 1,...,d$, we denote by $e_{\ell i}\in \{-1,1\}$ as
	i.i.d. Rademacher random variables.
	Then for each $\ell=1,...,d$ we define the following random functions:
	\begin{equation}
	f_{\ell}(\btheta) := \sum\limits_{i=0}^{M-1} e_{\ell i} \varepsilon^s \tilde{\varphi}\Big(\frac{\btheta - \btheta_i}{\varepsilon}\Big),~\btheta \in \RR^d.
	\end{equation}
	It is easy to see that $f_{\ell}(\btheta_i) =  a \varepsilon^s$ with $e_{\ell i}=1$ and $f_{\ell}(\btheta_i) =  -a \varepsilon^s$ with $e_{\ell i}=-1$, which implies that
	\begin{equation}
	e_{\ell i} = {\rm sign}(f_{\ell}(\btheta_i)),~\forall~i=0,...,M-1,~{\rm and}~\ell=1,...,d.
	\end{equation}
	Given $\varphi$ is compactly supported in $[-1,1]$, the functions $\varepsilon^s\tilde{\varphi}((\btheta-\btheta_i)/\varepsilon)$, $i=0,...,M-1$ have disjoint supports and are supported in the ball centered at $\btheta_i$ with radius $\varepsilon$.
	This further implies that $\|f_{\ell}\|_{s,\infty,1} \leq B$ given $\|\tilde{\varphi}\|_{s,\infty,1} \leq B$, and $0<\varepsilon \leq 1/8$.
	
	For now, we assume that for some $\delta >0$
	\begin{equation}
	\label{lowerbound2: assumption1}
	\inf_T \sup\limits_{\|f\|_{s,\infty,1}\leq B}\sup\limits_{\|\btheta\|\leq 1} \EE_{\btheta} (T(\bx) - f(\btheta))^2  < \delta^2,
	\end{equation}
	which immediately implies that
	\begin{equation}
	\label{lowerbound2: assumption2}
	\inf_T  \max\limits_{1\leq \ell \leq d} \max_{\btheta \in \bTheta} \EE_{\btheta} (T(\bx) - f_{\ell}(\btheta))^2  < \delta^2.
	\end{equation}
	This essentially means that for each $\ell=1,...,d$, there exists an estimator $T_{\ell}(\bx)$ such that
	\begin{equation}
	\label{assumption: upper}
	\max_{\btheta \in \bTheta} \EE_{\btheta} (T_{\ell}(\bx) - f_{\ell}(\btheta))^2  < \delta^2,~\ell=1,...,d,
	\end{equation}
	which leads to
	\begin{equation}
	\label{average_risk}
	\frac{1}{M} \sum_{i=0}^{M-1} \sum_{\ell=1}^d \EE_{\btheta_i} (T_{\ell}(\bx) - f_{\ell}(\btheta_i))^2 < d\delta^2.
	\end{equation}
	We denote by $\hat{T}(\bx) := (T_1(\bx),...,T_{d}(\bx))^T \in \RR^d$, and $\hat{f}(\btheta):=(f_1(\btheta),...,f_d(\btheta))\in \RR^d$, then we can rewrite (\ref{assumption: upper}) as
	\begin{equation}
	\frac{1}{M} \sum_{i=0}^{M-1} \EE_{\btheta_i} \big\| \hat{T}(\bx) - \hat{f}(\btheta_i)\big\|^2 < d\delta^2.
	\end{equation}
	
	Now we switch to consider the estimation problem of the random vector $\hat{f}(\btheta)\in\RR^d$ over the parameter space $\bTheta$ based on the observation $\bx$ with the prior
	$\Pi:=(\btheta,\be)$, where $\be \in \{-1,1\}^{d\times M}$ and is independent of $\btheta$. We further assume that $\btheta$ is uniformly distributed in $\bTheta$, and the entries of $\be$, $e_{\ell i}$ are i.i.d.
	Rademacher random variables for all $\ell=1,...,d$ and $i=1,...,M$.
	
	Given the prior $\Pi$, each entry of $\hat{f}$ only takes values in $\{a\varepsilon^s, -a\varepsilon^s \}$.
	Condition on $\bx$, we further assume that $\PP\{\hat{f}_{\ell} = a\varepsilon^s | \bx \} = p_{\ell}(\bx)$. Due to the independence of the entries of $\be$,
	we can define the Bayes estimator of $\hat{f}$ as $\EE[\hat{f}|\bx]$ with $\ell$-th entry as
	\begin{equation}
	\EE[\hat{f}_{\ell}| \bx] = a\varepsilon^s p_{\ell}(\bx) - a\varepsilon^s(1-p_{\ell}(\bx)).
	\end{equation}
	We denote by $R_{\Pi}(T)$ the average risk of an estimator $T$ with respect to the prior $\Pi$. Then due to (\ref{average_risk}) and the definition of Bayes estimator, we have
	\begin{equation}
	\label{bound: upper_bayes}
	R_{\Pi}(\EE[\hat{f}|\bx]) \leq 	R_{\Pi}(\hat{T}) < d\delta^2.
	\end{equation}
	On the other hand,
	\begin{equation}
	\label{formula: bayes_risk}
	R_{\Pi}(\EE[\hat{f}|\bx]):=\EE_{\bx} \big[\EE_{\Pi} \big\| \EE[\hat{f}|\bx]-\hat{f} \big\|^2 | \bx\big] = 4\varepsilon^{2s} \sum_{\ell=1}^d \EE_{\bx} \big[ \EE_{\Pi} [p_{\ell}(\bx)(1-p_{\ell}(\bx))|\bx] \big].
	\end{equation}
	Now we consider the quantity $\EE_{\Pi} [p_{\ell}(\bx)(1-p_{\ell}(\bx))|\bx]$. We denote by $\phi_i:=p(\bx| \btheta_i,\sigma^2\bI_d )$ the multivariate Gaussian density with mean $\btheta_i$ and covariance matrix
	$\sigma^2\bI_d$. Then
	\begin{equation}
	\label{equation: bayes}
	\begin{aligned}
	\EE_{\Pi} [p_{\ell}(\bx)(1-p_{\ell}(\bx))|\bx] &= \EE_{\Pi} \Big[ \frac{(\sum_{i=0}^M \mathbf{1}(e_{\ell i }=1)\phi_i)(\sum_{i'=0}^M \mathbf{1}(e_{\ell i' }=-1)\phi_{i'}) }{(\sum_{i=0}^M \phi_i)^2}\Big] \\
	&= \EE_{\Pi} \Big[ \frac{(\sum_{i,i'=0,i\neq i'}^M \mathbf{1}(e_{\ell i }=1) \mathbf{1}(e_{\ell i' }=-1)\phi_i\phi_{i'} }{(\sum_{i=0}^M \phi_i)^2}\Big] \\
	&= \frac{1}{4}\EE_{\Pi} \Big[ \frac{(\sum_{i,i'=0,i\neq i'}^M  \mathbf{1}(i\neq i') \phi_i\phi_{i'} }{(\sum_{i=0}^M \phi_i)^2}\Big] \\
	& = \frac{1}{4}\EE_{\Pi} \Big[ 1 - \frac{\sum_{i}^M   \phi_i^2  }{(\sum_{i=0}^M \phi_i)^2}\Big],
	\end{aligned}
	\end{equation}
	where the third equality is due to the fact that $\EE_{\Pi}[  \mathbf{1}(e_{\ell i }=1) \mathbf{1}(e_{\ell i' }=-1) | \bx] = \mathbf{1}\{ i\neq i'\}/4$ given $e_{\ell i}$'s are i.i.d. Rademacher random variables.
	Plug (\ref{equation: bayes}) into (\ref{formula: bayes_risk}), we obtain
	\begin{equation}
	R_{\Pi}(\EE[\hat{f}|\bx]) = d\varepsilon^{2s} \EE_{\bx} \Big[\EE_{\Pi} \Big[ 1 - \frac{\sum_{i}^M   \phi_i^2  }{(\sum_{i=0}^M \phi_i)^2} \big| \bx \Big] \Big] .
	\end{equation}
	Note that $\EE_{\Pi} \Big[ 1 - \frac{\sum_{i}^M   \phi_i^2  }{(\sum_{i=0}^M \phi_i)^2} \big| \bx \Big]$ is a function of the random variable $\bx$. The density function of the marginal distribution of $\bx$ is given by
	\begin{equation}
	p(\bx) = \frac{1}{M} \sum_{i=0}^M \phi_i.
	\end{equation}
	As a result,
	\begin{equation}
	\label{formula: bayes_risk_lower}
	\begin{aligned}
	\EE_{\bx} \Big[\EE_{\Pi} \Big[\frac{\sum_{i}^M   \phi_i^2  }{(\sum_{i=0}^M \phi_i)^2} \big| \bx \Big] \Big] & \leq \EE_{\bx} \Big [\frac{\max_{i} \phi_i}{\sum_{i=0}^M \phi_i}  \Big]
	= M^{-1}\int_{\RR^d}  \max_{i} \phi_i d\bx \\
	& = M^{-1}\int_{\RR^d} \frac{1}{(\sqrt{2\pi\sigma})^d} e^{\frac{-\min_{i} \|\bx - \btheta_i\|^2}{2\sigma^2}} d\bx  \\
	& \leq  M^{-1}\int_{\RR^d} \frac{1}{(\sqrt{2\pi\sigma})^d} e^{\frac{-[(\|\bx\| - 8\varepsilon)_+]^2}{2\sigma^2}} d\bx \\
	& =   M^{-1}\int_{\RR^d} \exp \{ (\|\by\|^2-[(\|\by\| - 8\varepsilon/\sigma)_+]^2)/2 \} p(\by | \mathbf{0}, \bI_d )d\by \\
	& \leq M^{-1}\int_{\RR^d} \exp \{ 8\varepsilon \|\by\|/\sigma  \vee 32\varepsilon^2/\sigma^2 \} p(\by | \mathbf{0}, \bI_d )d\by \\
	& \leq M^{-1} (e^{16\varepsilon\sqrt{d}/\sigma} \vee e^{32\varepsilon^2/\sigma^2}).
	\end{aligned}
	\end{equation}
	where $(x)_+ :=\max\{x,0 \}$ and $p(\by | \mathbf{0}, \bI_d )$ denotes the density of an isotropic Gaussian random vector. The third line is due to
	$\min_{i} \|\bx - \btheta_i\| \geq \min_i | \|\bx\| - \|\btheta_i\||$ and $\|\btheta_i\| = 8\varepsilon$. Set $\varepsilon = \sqrt{d/n}/\beta $ with $\beta = \max\{16/\log1.5, \sqrt{32/\log1.5}  \}$. Then from (\ref{formula: bayes_risk_lower}), we have
	\begin{equation}
	\label{bound: upper_bayes_risk}
	\EE_{\bx} \Big[\EE_{\Pi} \Big[\frac{\sum_{i}^M   \phi_i^2  }{(\sum_{i=0}^M \phi_i)^2} \big| \bx \Big] \Big] \leq \Big(\frac{3}{4}\Big)^d \leq \frac{3}{4},~\forall~d\geq 1.
	\end{equation}
	Combine (\ref{bound: upper_bayes_risk}) and (\ref{equation: bayes}), we obtain
	\begin{equation}
	\label{bound: lower_bayes}
	R_{\Pi}(\EE[\hat{f}|\bx]) \geq \frac{1}{4} d\varepsilon^{2s}.
	\end{equation}
	Taking $\delta:= \varepsilon^{s}/2$, (\ref{bound: lower_bayes}) contradicts (\ref{bound: upper_bayes}), which means that for $d\geq 1$, we have
	\begin{equation}
	\inf_T \sup\limits_{\|f\|_{s,\infty,1}\leq B}\sup\limits_{\|\btheta\|\leq 1} \EE_{\btheta} (T(\bx) - f(\btheta))^2  \geq  \delta^2 = \frac{\varepsilon^{2s}}{4} \gtrsim \Big(\frac{d}{n}\Big)^s.
	\end{equation}
    Especially, when $d = n^{\alpha}$ with some $\alpha\in (0,1)$, we obtain
    \begin{equation}
    	\inf_T \sup\limits_{\|f\|_{s,\infty,1}\leq B}\sup\limits_{\|\btheta\|\leq 1} \EE_{\btheta} (T(\bx) - f(\btheta))^2  \gtrsim n^{-s(1-\alpha)}.
    \end{equation}
	This completes the proof of Theorem~\ref{theorem: minimax_lower_2}.
	
\end{proof}

\subsection{Proof of Theorem~\ref{theorem: minimax_lower_1}}

\begin{proof}
	In order to simplify the presentation, we will continue to use some of the notations already defined in the proof of Theorem~\ref{theorem: minimax_lower_2}.
	For $i\in \{0,1,...,M-1\}$ and $\ell = 1,...,d$, we denote by $b_{\ell}(i)\in \{0,1\}$ as the $\ell$-th binary digit of $i$ so that $i=\sum^d_{\ell=1} b_{\ell}(i)2^{d-\ell}$.
	Similarly, we consider the following candidate functions:
	\begin{equation}
	f_{\ell}(\btheta) := \sum\limits_{i=0}^{M-1} (2b_{\ell}(i)-1) \varepsilon \tilde{\varphi}\Big(\frac{\btheta - \btheta_i}{\varepsilon}\Big),~\btheta \in \RR^d.
	\end{equation}
	Note that $f_{\ell}(\btheta_i) =  a \varepsilon$ with $b_{\ell}(i)=1$ and $f_{\ell}(\btheta_i) =  -a \varepsilon$ with $b_{\ell}(i)=0$, which implies that
	\begin{equation}
	b_{\ell}(i) = \frac{1+{\rm sign}(f_{\ell}(\btheta_i))}{2},~i=0,...,M-1,~and~\ell=1,...,d.
	\end{equation}
	Given $\varphi$ is compactly supported in $[-1,1]$, the functions $\varepsilon\tilde{\varphi}((\btheta-\btheta_i)/\varepsilon)$, $i=0,...,M-1$ have disjoint supports, which implies that $\|f_{\ell}\|_{s,\infty,1} \leq B$ due to  $\|\tilde{\varphi}\|_{s,\infty,1} \leq B$ and $\varepsilon \leq 1/8$.
	
	For now, we assume that for some $\delta >0$
	\begin{equation}
	\label{lowerbound2: assumption1}
	\inf_T \sup\limits_{\|f\|_{s,\infty,1}\leq B}\sup\limits_{\|\btheta\|\leq 1} \EE_{\btheta} (T(\bx) - f(\btheta))^2  < \delta^2,
	\end{equation}
	which immediately implies that
	\begin{equation}
	\label{lowerbound2: assumption2}
	\inf_T  \max\limits_{1\leq \ell \leq d} \max_{\btheta \in \Theta} \EE_{\btheta} (T(\bx) - f_{\ell}(\btheta))^2  < \delta^2.
	\end{equation}
	This essentially means that for each $\ell=1,...,d$, there exists an estimator $T_{\ell}(\bx)$ such that
	\begin{equation}
	\max_{\btheta \in \bTheta} \EE_{\btheta} (T_{\ell}(\bx) - f_{\ell}(\btheta))^2  < \delta^2,~\ell=1,...,d.
	\end{equation}
	By Markov's inequality, we get
	\begin{equation}
	\max_{\btheta \in \bTheta} \PP_{\btheta} \Big\{  |T_{\ell}(\bx) - f_{\ell}(\btheta)| \geq \frac{a \varepsilon}{2}  \Big\} \leq \frac{4\delta^2}{ a^2 \varepsilon^{2}}.
	\end{equation}
	Take $\delta^2:= a^2\varepsilon^{2}/16$, we have
	\begin{equation}
	\label{inequality: prob1}
	\max_{\btheta \in \bTheta} \PP_{\btheta} \Big\{  |T_{\ell}(\bx) - f_{\ell}(\btheta)| \geq \frac{a \varepsilon}{2}  \Big\} \leq \frac{1}{4}.
	\end{equation}
	Denote the event $\mathcal{E}$ as
	$$
	\mathcal{E} := \Big\{    |T_{\ell}(\bx) - f_{\ell}(\btheta_i)| <  \frac{a \varepsilon}{2}  \Big\}.
	$$
	On this event, we have ${\rm sign}(T_{\ell}(\bx)) = {\rm sign}(f_{\ell}(\btheta_i))$, for $\ell=1,...,d$.
	Therefore for $i=0,...,M-1$
	\begin{equation}
	\label{inequality: sign}
	\PP_{\btheta_i} \{{\rm sign}(T_{\ell}(\bx)) \neq {\rm sign}(f_{\ell}(\btheta_i)) \}\leq  \PP_{\btheta_i} \Big\{  |T_{\ell}(\bx) - f_{\ell}(\btheta_i)| \geq \frac{a \varepsilon}{2}  \Big\} \leq \frac{1}{4}.
	\end{equation}
	
	We define
	\begin{equation*}
	\hat{\omega}:= ((1+{\rm sign}(T_1(\bx)))/2,...,(1+{\rm sign}(T_d(\bx)))/2)^T \in \{0,1\}^d;
	\end{equation*}
	and for $i=0,1,...,M-1$
	\begin{equation}
	\omega(\btheta_i):= (b_1(i),...,b_d(i))^T.
	\end{equation}
	Let $\Lambda:=\{ 0,1 \}^d$ be the set of all binary sequences of length $d$, then it is easy to check that $\Lambda = \{\omega(\btheta): \btheta \in \bTheta \}$.
	Let $\{ P_{\btheta}: \btheta \in \bTheta \}$ be a set of $2^d$ Gaussian measures with mean $\btheta_i$ and covariance matrix $\sigma^2\bI_d$.
	We state a user-friendly version of Assouad's Lemma as follows:
	\begin{lemma}[Assouad's Lemma]
		\label{lemma: Assouad}
		If the KL divergence $K(P_{\btheta_i} || P_{\btheta_j})\leq \alpha <\infty$ for any $\omega(\btheta_i)$, $\omega(\btheta_j)\in \Lambda$ with $\rho(\omega(\btheta_i),\omega(\btheta_j))=1$.
		Let $\EE_{\btheta_i}$ denote the corresponding expectation to the probability measure $P_{\btheta_i}$. Then
		\begin{equation}
		\inf_{\hat{\omega}}\max_{\omega(\btheta_i) \in \Lambda} \EE_{\btheta_i} \rho(\hat{\omega},\omega(\btheta_i)) \geq \frac{d}{2} \max \Big\{\frac{1}{2}e^{-\alpha}, 1-\sqrt{\alpha/2} \Big\}
		\end{equation}
		where $\rho$ denotes the Hamming distance of the binary sequences.
	\end{lemma}
	On one hand,
	\begin{equation}
	\label{contradiction: 21}
	\begin{aligned}
	\EE_{\btheta_i} \rho(\hat{\omega},\omega(\btheta_i)) &= \sum_{\ell=1}^d \PP_{\btheta_i}\{ \hat{\omega}_{\ell} \neq \omega_{\ell}(\btheta_i) \} = \sum_{\ell=1}^d \PP_{\btheta_i}\{ {\rm sign}(T_{\ell}(\bx))\neq {\rm sign}(f_{\ell}(\btheta_i))\} \leq d/4 .
	\end{aligned}
	\end{equation}
	where the inequality is due to (\ref{inequality: sign}).
	On the other hand, take $\varepsilon = \sigma/24$, for any $\btheta_i$, $\btheta_j\in \Lambda$,
	\begin{equation}
	K(P_{\btheta_i} || P_{\btheta_j})= \frac{1}{2} \langle \bSigma^{-1} (\btheta_i - \btheta_j) , (\btheta_i - \btheta_j) \rangle \leq  \frac{128  \varepsilon^2}{\sigma^2}  \leq \frac{2}{9}<\infty.
	\end{equation}
	By Lemma~\ref{lemma: Assouad}, we have
	\begin{equation}
	\inf_{\hat{\omega}}\max_{\omega(\btheta_i) \in \Lambda} \EE_{\btheta_i} \rho(\hat{\omega},\omega(\btheta_i)) \geq \frac{d}{3},
	\end{equation}
	which contradicts (\ref{contradiction: 21}). As a consequence, bound (\ref{lowerbound2: assumption1}) does not hold for $\delta^2 = a^2\varepsilon^{2}/16$ and $\varepsilon =  \sigma/24$.
	This means that for some numerical constant $c_1$ we have
	\begin{equation*}
	\inf_T \sup\limits_{\|f\|_{s,\infty,1}\leq 1}\sup\limits_{\|\btheta\|\leq 1} \EE_{\btheta} (T(\bx) - f(\btheta))^2 \geq c_1 \sigma^2.
	\end{equation*}
    Setting $\sigma^2 = n^{-1}$ yields
    \begin{equation*}
    	\inf_T \sup\limits_{\|f\|_{s,\infty,1}\leq B}\sup\limits_{\|\btheta\|\leq 1} \EE_{\btheta} (T(\bx) - f(\btheta))^2 \geq c_1 n^{-1}.
    \end{equation*}
\end{proof}

\subsection{Proof of Theorem~\ref{theorem: efficiency}}
\begin{proof}
	The main idea of the proof is based on an application of van Trees inequality, see~\cite{gill1995applications} and a method developed in Theorem 2.4 of~\cite{koltchinskii2018efficient}.
	Firstly, we prove the following lemma.
	\begin{lemma}
		\label{Lemma: lowerbound_1}
		For all $\btheta \in \RR^d$ such that $\|\btheta - \btheta_0\| \leq c n^{-1/2}$, and for $f\in \calF^s(\RR^d)$ with $s>1$ and condition (\ref{variance: size}) satisfied. Then with some numerical constant $C_1>0$ the following bound holds
		\begin{equation}
		\Big| \frac{ \sigma^2_{f,\bepsilon}(\btheta)}{ \sigma^2_{f,\bepsilon}(\btheta_0)} -1 \Big| \leq C_1 c \tau^2\cdot  n^{-(2\wedge s)/2 + 1/2}
		\end{equation}
	\end{lemma}
	The bound of Lemma~\ref{Lemma: lowerbound_1} implies that
	\begin{equation}
	\label{bound: lower_switch1}
	\begin{aligned}
	&\sup\limits_{\btheta\in \calB(\btheta_0;cn^{-1/2})} \frac{\EE_{\btheta} (T(\bx) - f(\btheta))^2 }{ \sigma^2_{f,\bepsilon}(\btheta)} = \sup\limits_{\btheta\in \calB(\btheta_0;cn^{-1/2})} \frac{\EE_{\btheta} (T(\bx) - f(\btheta))^2 }{ \sigma^2_{f,\bepsilon}(\btheta_0)} \cdot \frac{ \sigma^2_{f,\bepsilon}(\btheta_0)}{ \sigma^2_{f,\bepsilon}(\btheta)} \\
	& \geq \sup\limits_{\btheta\in \calB(\btheta_0;cn^{-1/2})} \frac{\EE_{\btheta} (T(\bx) - f(\btheta))^2 }{ \sigma^2_{f,\bepsilon}(\btheta_0)}  \frac{1}{1 +c \tau^2\cdot  n^{-(2\wedge s)/2 + 1/2}}.
	\end{aligned}
	\end{equation}
	
	Now we switch to bound
	\begin{equation*}
	\sup\limits_{\btheta\in \calB(\btheta_0;cn^{-1/2})} \frac{\EE_{\btheta} (T(\bx) - f(\btheta))^2 }{ \sigma^2_{f,\bepsilon}(\btheta_0)} .
	\end{equation*}
	Set $c_0:=c/\tau$, then for any $t\in [-c_0, c_0]$ and $\bdelta \in \RR^d$, we define
	\begin{equation}
	\btheta_t := \btheta_0 + t\bdelta.
	\end{equation}
	Consider the estimation of the following functional
	\begin{equation*}
	\varphi(t) := f(\btheta_t),~t\in [-c_0,c_0]
	\end{equation*}
	based on an observation $\bx\sim \mathcal{N}(\btheta_t;n^{-1}\bSigma)$ . By choosing $\bdelta:= n^{-1}\bSigma \nabla f(\btheta_0)/ \sigma_{f,\bepsilon}(\btheta_0)$, we have
	\begin{equation}
	\begin{aligned}
	\|t\bdelta\| &\leq \frac{c_0n^{-1} \|\bSigma\|_{op} \|\nabla f(\btheta_0)\| }{ \sigma_{f,\bepsilon}(\btheta_0)} \leq \frac{c_0 n^{-1}\|\bSigma\|_{op}  \|f\|_{s,\infty,1} }{ \sigma_{f,\bepsilon}(\btheta_0)}\leq c_0  \tau n^{-1/2} \leq c n^{-1/2} < 1,
	\end{aligned}
	\end{equation}
	which implies that $\btheta_t \in \calB(\btheta_0;cn^{-1/2})$. As a consequence,
	\begin{equation}
	\label{bound: lower_switch2}
	\sup\limits_{\btheta \in \calB(\btheta_0;cn^{-1/2})} \frac{\EE_{\btheta} (T(\bx) - f(\btheta))^2 }{ \sigma^2_{f,\bepsilon}(\btheta_0)} \geq \sup\limits_{t \in [-c_0 ,c_0]} \frac{\EE_{\btheta} (T(\bx) - \varphi(t))^2 }{ \sigma^2_{f,\bepsilon}(\btheta_0)}.
	\end{equation}
	Let $\pi$ be a prior density on $[-1,1]$ with $\pi(-1) = \pi(1) = 0$ and such that
	\begin{equation}
	J_{\pi} := \int^1_{-1} \frac{(\pi'(s))^2}{\pi(s)} ds < \infty.
	\end{equation}
	Denote $\pi_{c_0}(t) = c_0^{-1}\pi(t/c_0)$, with $t\in[-c_0, c_0]$. Then $J_{\pi_{c_0}} =  J_{\pi}/c_0^2$.
	
	By the van Trees inequality~\cite{gill1995applications}, for any estimator $T(\bx)$ of $\varphi(t)$, the following inequalities hold
	\begin{equation}
	\label{bound: key_inequality}
	\begin{aligned}
	& \sup\limits_{t\in[-c_0,c_0]} \EE_{t} (T(\bx) - \varphi(t))^2 \geq \int^{c_0}_{-c_0} \EE_{t} (T(\bx) - \varphi(t))^2 \pi_{c_0}(t) dt  \\
	& \geq \frac{\big(\int^{c_0}_{-c_0} \varphi'(t)\pi_{c_0}(t)dt\big)^2}{\int^{c_0}_{-c_0} \mathcal{I}(t) \pi_{c_0}(t)dt + J_{\pi}/c_0^2} \geq \frac{\big(\int^{c_0}_{-c_0} \varphi'(t)\pi_{c_0}(t)dt\big)^2}{1+ J_{\pi}/c_0^2}.
	\end{aligned}
	\end{equation}
	The last inequality is due to the fact that when $\bdelta = n^{-1}\bSigma \nabla f(\btheta_0)/ \sigma_{f,\bepsilon}(\btheta_0)$, the Fisher information $\mathcal{I}(t) = n \langle \bSigma^{-1} \bdelta, \bdelta \rangle = 1$.
	It remains to give a lower bound on $\big(\int^{c_0}_{-c_0} \varphi'(t)\pi_{c_0}(t)dt\big)^2$. Recall that $\varphi'(t) = \langle \bdelta ,  \nabla f(\btheta_t)\rangle$ and let
	\begin{align}
	& {\rm I}_0 := \int^{c_0}_{-c_0} \varphi'(0) \pi_{c_0}(t) dt = \int^{c_0}_{-c_0} \langle \bdelta, \nabla f(\btheta_0) \rangle \pi_{c_0}(t) dt = \langle \bdelta, \nabla f(\btheta_0)\rangle \\
	& {\rm I}_1 := \int^{c_0}_{-c_0} (\varphi'(t) - \varphi'(0)) \pi_{c_0}(t) dt .
	\end{align}
	Then we have
	\begin{equation}
	\Big(\int^{c_0}_{-c_0} \varphi'(t)\pi_{c_0}(t)dt\Big)^2 = ({\rm I}_0 + {\rm I}_1)^2 \geq {\rm I}_0^2 - 2|{\rm I}_0||{\rm I}_1| =  \sigma^2_{f,\bepsilon} (\btheta_0) - 2 \sigma_{f,\bepsilon}(\btheta_0) |{\rm I}_1|,
	\end{equation}
	where we used the assumption that $\bdelta = n^{-1}\bSigma \nabla f(\btheta_0)/\sigma_{f,\bepsilon}(\btheta_0)$. Now we turn to bound $|{\rm I}_1|$. Similar to the proof of Lemma~\ref{Lemma: lowerbound_1},
	\begin{equation}
	\begin{aligned}
	|{\rm I}_1| \leq |\varphi'(t) - \varphi'(0)| &=  |\langle \bdelta , \nabla f(\btheta_t) - \nabla f(\btheta_0)\rangle | \leq \|\bdelta \| \|\nabla f(\btheta_t) - \nabla f(\btheta_0)\| \\
	& \leq \frac{n^{-1/2}\|\bSigma\|_{op}\|\nabla f(\btheta_0)\|}{ \sigma_{f,\bepsilon}(\btheta_0)} \cdot   \frac{\|f\|_{s,\infty,1}}{n^{(2\wedge s)/2}} \\
	& \lesssim  c \tau^2 \sigma_{f,\bepsilon}(\btheta_0)n^{-(2\wedge s)/2+1/2}\\
	\end{aligned}
	\end{equation}
	As a result, we have
	\begin{equation}
	\label{bound: lower_numerator}
	\Big(\int^{c_0}_{-c_0} \varphi'(t)\pi_{c_0}(t)dt\Big)^2  \geq \sigma^2_{f,\bepsilon}(\btheta_0) \Big( 1 -  c\tau^2 n^{-(2\wedge s)/2+1/2} \Big).
	\end{equation}
	By plugging (\ref{bound: lower_numerator}) into (\ref{bound: key_inequality}), we get
	\begin{equation}
	\sup\limits_{t\in[-c_0,c_0]} \frac{\EE_{t} (T(\bx) - \varphi(t))^2 }{ \sigma^2_{f,\bepsilon}(\btheta_0) }  \geq  \frac{\big( 1 - c\tau^2 n^{-(2\wedge s)/2+1/2}\big) }{ (1 + J_{\pi}/c_0^2) }
	\end{equation}
	Together with (\ref{bound: lower_switch1}) and (\ref{bound: lower_switch2}), we get
	\begin{equation}
	\begin{aligned}
	\inf_T  \sup\limits_{\|\btheta -\btheta_0\|\leq cn^{-1/2}} \frac{\EE_{\btheta} (T(\bx) - f(\btheta))^2}{ \sigma^2_{f,\bepsilon}(\btheta)} & \geq  \frac{\big( 1 - c\tau^2 n^{-(2\wedge s)/2+1/2}\big) }{(1 + c\tau^2 n^{-(2\wedge s)/2+1/2}) (1 + J_{\pi}/c_0^2) } \\
	& \geq 1 - C_1 \tau^2 \big(  cn^{-(2\wedge s)/2+1/2} + c^{-2} \big)
	\end{aligned}
	\end{equation}
	for some constant $C_1$.
\end{proof}

\subsection{Proof of Theorem~\ref{theorem: adaptive}}
\begin{proof}
	We consider
	\begin{equation}
	\label{difference: g}
	\begin{aligned}
	\hat{g}(\bar{\bx}) - g(\bar{\bx}) & =  \frac{1}{(2\pi)^{d/2}} \Big(  \int_{\hat{\Omega}}  \mathcal{F} f (\bzeta)e^{\langle \hat{\bSigma} \bzeta, \bzeta \rangle/2n} e^{i\bzeta \cdot \bx} d\bzeta
	-\int_{\Omega}  \mathcal{F} f(\bzeta)e^{\langle \bSigma \bzeta, \bzeta \rangle/2n} e^{i\bzeta \cdot \bx} d\bzeta \Big) \\
	& =  \frac{1}{(2\pi)^{d/2}} \Big(  \int_{\Omega}  \mathcal{F} f(\bzeta) \big( e^{\langle (\hat{\bSigma}-\bSigma) \bzeta, \bzeta \rangle/2n}-1 \big) e^{\langle \bSigma \bzeta, \bzeta \rangle/2n} e^{i\bzeta \cdot \bx} d\bzeta  \Big) \\
	& + \frac{1}{(2\pi)^{d/2}}  \int_{(\hat{\Omega} \backslash \Omega) \cup (\Omega \backslash \hat{\Omega})} \mathcal{F} f (\bzeta)e^{\langle \hat{\bSigma} \bzeta, \bzeta \rangle/2n} e^{i\bzeta \cdot \bx} d\bzeta.
	\end{aligned}
	\end{equation}
	In the following, we will bound the two terms respectively. We denote by $\Delta: = e^{\langle (\hat{\bSigma}-\bSigma) \bzeta, \bzeta \rangle/2n}$.  By the concentration bound on the Gaussian covariance matrix, see~\cite{koltchinskii2017}, for any $t>0$ we get with probability at least $1-e^{-t}$
	\begin{equation}
	\label{event: 1}
	\big|  \langle (\hat{\bSigma} -\bSigma)\bzeta , \bzeta \rangle \big| \leq \big\| \hat{\bSigma} -\bSigma \big\|_{op} \big\| \bzeta \big\|^2 \lesssim  \|\bSigma\|_{op}\Big( \sqrt{\frac{\br(\bSigma)}{n}} \bigvee \sqrt{\frac{t}{n}} \Big) \big\| \bzeta \big\|^2,
	\end{equation}
	and
	\begin{equation}
		\label{event: 2}
		\EE \big|  \langle (\hat{\bSigma} -\bSigma)\bzeta , \bzeta \rangle \big| \leq \EE \big\| \hat{\bSigma} -\bSigma \big\|_{op} \big\| \bzeta \big\|^2 \lesssim 	\|\bSigma\|_{op} \cdot  \sqrt{\frac{\br(\bSigma)}{n}}  \cdot \big\| \bzeta \big\|^2.
	\end{equation}
	We denote by $\calE_1$ as the event in (\ref{event: 1}). Conditionally on $\calE_1$,
	when $n$ is large enough, we have
	\begin{equation}
	\label{bound: part1}
	\begin{aligned}
	& \Big| \frac{1}{(2\pi)^{d/2}} \Big(  \int_{\Omega}  \mathcal{F} f(\bzeta) \big( e^{\langle (\hat{\bSigma}-\bSigma) \bzeta, \bzeta \rangle/2n}-1 \big) e^{\langle \bSigma \bzeta, \bzeta \rangle/2n} e^{i\bzeta \cdot \bx} d\bzeta  \Big) \Big|  \\
	&\leq \frac{1}{(2\pi)^{d/2}}   \int_{\Omega}  \Big| \mathcal{F} f(\bzeta) \big( e^{\langle (\hat{\bSigma}-\bSigma) \bzeta, \bzeta \rangle/2n}-1 \big) e^{\langle \bSigma \bzeta, \bzeta \rangle/2n} \Big| d\bzeta  \\
	& \leq \frac{e^2}{(2\pi)^{d/2}} \Big(  \int_{\Omega}  \Big| \mathcal{F} f(\bzeta)\Big| \cdot \big| \Delta-1 \big| d\bzeta  \\
	& \lesssim \frac{e^2}{(2\pi)^{d/2}} \Big(  \int_{\Omega}  \Big| \mathcal{F} f(\bzeta)\Big| \cdot \frac{\|\bSigma\|_{op}}{n} \Big( \sqrt{\frac{\br(\bSigma)}{n}} \bigvee \sqrt{\frac{t}{n}} \Big) \big\| \bzeta \big\|^2 d\bzeta  \\
	&  \lesssim \frac{\big\| f \big\|_{s,\infty,1}}{n^{(2\wedge s)/2}} \Big( \sqrt{\frac{\br(\bSigma)}{n}} \bigvee \sqrt{\frac{t}{n}} \Big) .
	\end{aligned}
	\end{equation}
	Similarly, we have
	\begin{equation}
		\begin{aligned}
		& \EE \Big| \frac{1}{(2\pi)^{d/2}} \Big(  \int_{\Omega}  \mathcal{F} f(\bzeta) \big( e^{\langle (\hat{\bSigma}-\bSigma) \bzeta, \bzeta \rangle/2n}-1 \big) e^{\langle \bSigma \bzeta, \bzeta \rangle/2n} e^{i\bzeta \cdot \bx} d\bzeta  \Big) \Big|
		& \lesssim  \frac{\big\| f \big\|_{s,\infty,1}}{n^{(2\wedge s)/2}} \cdot  \sqrt{\frac{\br(\bSigma)}{n}} .
		\end{aligned}
	\end{equation}
	
	Now we switch to bound the second part on the right hand side of (\ref{difference: g}). We denote by $\hat{R}:= \sup_{\bzeta\in \hat{\Omega}} \|\bzeta\|$. Conditionally on $\calE_1$,
	\begin{equation}
	\begin{aligned}
		& \Big| \frac{1}{(2\pi)^{d/2}}  \int_{(\hat{\Omega} \backslash \Omega) \cup (\Omega \backslash \hat{\Omega})} \mathcal{F} f (\bzeta)e^{\langle \hat{\bSigma} \bzeta, \bzeta \rangle/2n} e^{i\bzeta \cdot \bx} d\bzeta \Big|\\
		&  \leq  \frac{1}{(2\pi)^{d/2}}  \int_{(\hat{\Omega} \backslash \Omega) \cup (\Omega \backslash \hat{\Omega})} \Big| \mathcal{F} f (\bzeta) \Big| e^{\langle \hat{\bSigma} \bzeta, \bzeta \rangle/2n}  d\bzeta  \\
		& \lesssim   \big\| f \big\|_{s,\infty,1} \cdot  (R - |\hat{R} - R|)^{-s}.
		\end{aligned}
	\end{equation}
	Recall that $R \gg 1$, and
	\begin{equation}
	\begin{aligned}
		\Big| R-\hat{R} \Big| &= \Big|  \sqrt{\frac{n}{\br (\bSigma)}} - \sqrt{\frac{n}{ \br ( \hat{\bSigma})}}\Big| .
		\end{aligned}
	\end{equation}
	As we can see that
	\begin{equation}
	\begin{aligned}
	\sqrt{{\rm tr}(\hat{\bSigma})}  &=  \sqrt{{\rm tr}(\sum_{j=1}^{n-1} \widetilde{\bbeta}_j \otimes \widetilde{\bbeta}_j)/(n-1)}
	& = \frac{1}{\sqrt{n-1}} \sqrt{\sum_{j=1}^{n-1} \big\| \widetilde{\bbeta}_j\big\|^2 }.
	\end{aligned}
	\end{equation}
	Suppose that $\{\widetilde{\bbeta}'_{j}\}_{j=1}^{n-1}$ is another independent identical copy of $\{\widetilde{\bbeta}_{j}\}_{j=1}^{n-1}$. Then we have
	\begin{equation}
	\begin{aligned}
	\Big| \sqrt{{\rm tr}(\hat{\bSigma})}  -   \sqrt{{\rm tr}(\hat{\bSigma}')} \Big|
	& =  \Big| \sqrt{ \frac{ {\rm tr}(\sum_{j=1}^{n-1} \widetilde{\bbeta}_j \otimes \widetilde{\bbeta}_j)}{n-1} } - \sqrt{ \frac{ {\rm tr}(\sum_{j=1}^{n-1} \widetilde{\bbeta}_j' \otimes \widetilde{\bbeta}_j')}{n-1} } \Big| \\
	& \leq \frac{1}{n-1} \Big| \sqrt{\sum_{j=1}^{n-1} \big\| \widetilde{\bbeta}_j\big\|^2 } - \sqrt{\sum_{j=1}^{n-1} \big\| \widetilde{\bbeta}'_j\big\|^2 } \Big| \\
	& \leq \frac{1}{n-1} \sqrt{\sum_{j=1}^{n-1} \big\| \widetilde{\bbeta}_j - \widetilde{\bbeta}_j'\big\|^2 } ,
	\end{aligned}
	\end{equation}
	which shows that $\sqrt{{\rm tr}(\hat{\bSigma})}$ is a Lipschitz function of $(\widetilde{\bbeta}_1,...,\widetilde{\bbeta}_{n-1})$.
	
	We use the following Gaussian concentration inequality (in a  bit non-standard fashion, see~\cite{koltchinskii2017}, Section 3 for a similar argument) as in Lemma~\ref{Lemma:isoperimetric}, which is a corollary of the classical Gaussian isoperimetric inequality, see~\cite{gine2016mathematical} chapter 2.
	\begin{lemma}
		\label{Lemma:isoperimetric}
		Let $X_1,...,X_d$ be i.i.d. centered Gaussian random variables in a Hilbert space $\mathbb{H}$ with covariance operator $\Sigma$. Let $f:\mathbb{R}^d \rightarrow \mathbb{R}$ be a function satisfying the following
		Lipschitz condition with Lipschitz constant $L>0$:
		$$
		\Big|f(x_1,...,x_d)-f(x'_1,...,x'_d) \Big| \leq L \Big(\sum\limits_{j=1}^d\|x_j-x'_j\|^2   \Big)^{1/2},~x_1,...,x_d,x'_1,...,x'_d\in \mathbb{H}.
		$$
		Suppose that, for a real number $M$,
		$$
		\mathbb{P}\{f(X_1,...,X_d) \geq M\} \geq 1/4~{\rm and}~\mathbb{P}\{f(X_1,...,X_d) \leq M\} \geq 1/4.
		$$
		Then, there exists a numerical constant $D$ such that for all $t\geq 1$,
		$$
		\mathbb{P}\Big\{ |f(X_1,...,X_d) - M| \geq DL\|\Sigma\|_{op}^{1/2}\sqrt{t} \Big\} \leq e^{-t}.
		$$
	\end{lemma}
	By Lemma~\ref{Lemma:isoperimetric}, we get
	for all $t\geq 1$, with probability at least $1-e^{-t}$
	\begin{equation}
	\Big| \sqrt{{\rm tr}(\hat{\bSigma})}  -   \sqrt{{\rm tr}(\bSigma)} \Big| \lesssim  \big\|\bSigma\big\|_{op}^{1/2}\sqrt{ \frac{ t}{n}},
	\end{equation}
	together with (\ref{event: 1}), it implies that with high probability for some $\beta >0$
	\begin{equation}
	\big| \hat{R} - R \big| = O(n^{-\beta}) .
	\end{equation}
    We denote by event $\calE_2: = \{| \hat{R} - R | \leq \eta(t)\}$ and $M:= (\hat{\Omega} -\Omega) \cup (\Omega - \hat{\Omega})$. Then on event $\calE_1$ and $\calE_2$, similarly as in
	(\ref{bound: part1}) we have
	\begin{equation}
	\label{bound: part2}
	\begin{aligned}
	& \Big| \frac{1}{(2\pi)^{d/2}}  \int_{(\hat{\Omega} -\Omega) \cup (\Omega - \hat{\Omega})} \mathcal{F} f(\bzeta)e^{\langle \hat{\bSigma} \bzeta, \bzeta \rangle/2n} e^{i\bzeta \cdot \bx} d\bzeta \Big|
	 \lesssim  \big\| f \big\|_{s,\infty,1} \cdot R^{-s}\\
	\end{aligned}
	\end{equation}
	given that $R\gg 1$.
	To sum up, we obtain for any $t>0$ with probability at least $1-e^{-t}$
	\begin{equation}
	\big| \hat{g}(\bar{\bx}) - g(\bar{\bx})\big| \lesssim \frac{\big\| f \big\|_{s,\infty,1}}{n^{(2\wedge s)/2}} \cdot \Big( \sqrt{\frac{\br(\bSigma)}{n}} \bigvee \sqrt{\frac{t}{n}} \Big)  + \big\| f \big\|_{s,\infty,1} \cdot R^{-s}.
	\end{equation}
\end{proof}

\subsection{Proof of Technical Lemmas}

\begin{proof}[Proof of Lemma~\ref{lemma: variance}.]
	Firstly note that
	\begin{equation}
		\begin{aligned}
			\big|\sigma_{f^N,\bepsilon}(\btheta) - \sigma_{f,\bepsilon}(\btheta)\big| &= n^{-1/2} \cdot \big| \sqrt{\langle \bSigma \nabla f^N(\btheta), \nabla f^N(\btheta)\rangle} -  \sqrt{\langle \bSigma \nabla f(\btheta), \nabla f(\btheta)\rangle} \big| \\
			& = n^{-1/2} \cdot \big| \|\bSigma^{1/2} \nabla f^N(\btheta)\|  -  \|\bSigma^{1/2} \nabla f(\btheta)\|\big| \\
			& \leq n^{-1/2} \cdot \| \bSigma^{1/2} (\nabla f^N (\btheta) - \nabla f(\btheta)) \|  \\
			& \leq n^{-1/2}  \| \bSigma\|_{op} ^{1/2} \| (\nabla f^N (\btheta) - \nabla f(\btheta)) \| .
		\end{aligned}
	\end{equation}
	On the other hand, by the definition of $\big\| f \big\|_{s,\infty,1}$
	\begin{equation}
		\begin{aligned}
			\big \| (\nabla f^N (\btheta) - \nabla f(\btheta)) \big\| & \leq (2\pi)^{-d/2} \int_{\RR^d \backslash \Omega} \big| \calF f(\bzeta)\big| \cdot  \| \bzeta\|d\bzeta \\
			&  \leq \big\| f \big\|_{s,\infty,1} \cdot R^{1-s}.
		\end{aligned}
	\end{equation}
	Therefore, given $R \leq \sqrt{n}$ and $\big\| f \big\|_{s,\infty,1} \cdot R^{-s} = o(n^{-1/2})$ when $s>1/(1-\alpha)$, we get
	\begin{equation}
		\sigma_{f^N,\bepsilon}(\btheta)  = \sigma_{f,\bepsilon}(\btheta) + o\big( n^{-1/2} \big).
	\end{equation}
\end{proof}

\begin{proof}[Proof of Lemma~\ref{Lemma: lowerbound_1}]
	\begin{equation}
		\begin{aligned}
			\Big| \frac{ \sigma^2_{f,\bepsilon}(\btheta)}{ \sigma^2_{f,\bepsilon}(\btheta_0)} -1 \Big| &= \frac{c n^{-1} \cdot |\langle \bSigma \nabla f(\btheta), \nabla f(\btheta)\rangle - \langle \bSigma \nabla f(\btheta_0), \nabla f(\btheta_0)\rangle|}{ \sigma^2_{f,\bepsilon}(\btheta_0)} \\
			& \leq \frac{cn^{-1} \cdot \|\bSigma\|_{op} \|\nabla f(\btheta) - \nabla f(\btheta_0)\| (\|\nabla f(\btheta)\| + \|\nabla f(\btheta_0)\|)}{ \sigma^2_{f,\bepsilon}(\btheta_0)}.
		\end{aligned}
	\end{equation}
	For the term $n^{-1/2}\cdot \|\nabla f(\btheta) - \nabla f(\btheta_0)\|$ we decompose it into two parts,
	\begin{equation}
		\label{bound: variance1}
		\begin{aligned}
			n^{-1/2} \cdot \|\nabla f(\btheta) - \nabla f(\btheta_0)\| & \leq (2\pi)^{-d/2} \int_{\Omega} \Big| \calF f(\bzeta) \Big| \cdot \Big| e^{i\bzeta^T\btheta} - e^{i\bzeta^T\btheta_0}\Big|\cdot \|\bzeta\|/\sqrt{n} d\bzeta \\
			& + (2\pi)^{-d/2} \int_{ \Omega^c} \Big| \calF f(\bzeta) \Big| \cdot \Big| e^{i\bzeta^T\btheta} - e^{i\bzeta^T\btheta_0}\Big|\cdot \|\bzeta\|/\sqrt{n} d\bzeta
		\end{aligned}
	\end{equation}
	where $\Omega:= \{ \bxi: \|\bxi\| \leq \sqrt{n} \}$.
	For the first term on the right hand side,
	\begin{equation}
		\begin{aligned}
			& (2\pi)^{-d/2} \int_{\Omega} \Big| \calF f(\bzeta) \Big| \cdot \Big| e^{i\bzeta^T\btheta} - e^{i\bzeta^T\btheta_0}\Big|\cdot \|\bzeta\|/\sqrt{n} d\bzeta \\
			&  \lesssim (2\pi)^{-d/2} \int_{\RR^d} \Big| \calF f(\bzeta) \Big| \cdot \|\bzeta\|^2/\sqrt{n} d\bzeta \cdot \|\btheta - \btheta_0 \| \\
			&  \lesssim (2\pi)^{-d/2} \int_{\RR^d} \Big| \calF f(\bzeta) \Big| \cdot \|\bzeta\|^2\cdot c n^{-1} d\bzeta  \\
			& \lesssim c \big\| f \big\|_{s,\infty,1} \cdot n^{-(2 \wedge s)/2}.
		\end{aligned}
	\end{equation}
	While for the second term, under the condition
	\begin{equation}
		\begin{aligned}
			& (2\pi)^{-d/2} \int_{ \Omega^c} \Big| \calF f(\bzeta) \Big| \cdot \Big| e^{i\bzeta^T\btheta} - e^{i\bzeta^T\btheta_0}\Big|\cdot \|\bzeta\|/\sqrt{n} d\bzeta \\
			& \leq 2 (2\pi)^{-d/2} \int_{ \Omega^c} \Big| \calF f(\bzeta) \Big| \cdot \|\bzeta\|/\sqrt{n} d\bzeta \\
			& \lesssim \big\| f \big\|_{s,\infty,1} \cdot n^{-s/2} .
		\end{aligned}
	\end{equation}
	Under the condition $n^{-1/2} \cdot \big\| f \big\|_{s,\infty,1} \leq \tau \sigma_{f,\bepsilon}(\btheta_0)$, we obtain
	\begin{equation}
		\frac{n^{-1/2} \cdot \|\nabla f(\btheta) - \nabla f(\btheta_0)\|}{\sigma_{f,\bepsilon}(\btheta_0)} \lesssim c\tau  n^{-(2\wedge s)/2 + 1/2}.
	\end{equation}
	On the other hand, a natural bound of $\|\nabla f(\btheta)\|$ is
	\begin{equation}
		\label{bound: variance2}
		\begin{aligned}
			\| \nabla f(\btheta) \| & = (2\pi)^{-d/2} \int_{\RR^d} \Big| \calF f(\bzeta) \Big| \cdot \|\bzeta\| d\bzeta \leq \big\| f \big\|_{s,\infty,1}  .
		\end{aligned}
	\end{equation}
	Combining the results in (\ref{bound: variance1}) and (\ref{bound: variance2}) yields
	\begin{equation*}
		\Big| \frac{ \sigma^2_{f,\bepsilon}(\btheta)}{ \sigma^2_{f,\bepsilon}(\btheta_0)} -1 \Big| \lesssim c \tau^2\cdot  n^{-(2\wedge s)/2 + 1/2} .
	\end{equation*}
\end{proof}

\bibliographystyle{plain}
\bibliography{refer}

\newpage
\appendix
\section{Appendix}

\begin{table}[!htb]
	\centering
	\begin{center}
		\begin{tabular}{|c|c|c|c|c|c|}
			\hline\hline
			$\alpha$-value & $d$-dimension & Smoothness Threshold & Plug-in Bias & TF Bias & Adaptive Bias\\ \hline
			0.40             &          40         &               1.6667     &   0.0043  &  0.0003  &  0.0003\\ \hline
			0.45             &          63         &               1.8182     &   0.0057  &  0.0003  &  0.0006 \\ \hline
			0.50             &         100         &               2.0000     &   0.0098  &  0.0011  &  0.0019 \\ \hline
			0.55             &         158         &               2.2222     &   0.0178  &  0.0014  &  0.0030 \\ \hline
			0.60             &         251         &               2.5000     &   0.0264  &  0.0034  &  0.0067  \\ \hline
			0.65             &         398         &               2.8571     &   0.0486  &  0.0031  &  0.0096 \\ \hline
			0.70             &         631         &               3.3333     &   0.0843  &  0.0075  &  0.0166 \\ \hline
			0.75             &         1000        &               4.0000     &   0.1515  &  0.0166  &  0.0380 \\ \hline
			0.80           &           1585        &               5.0000     &   0.3415  &  0.0191  &  0.0515 \\ \hline
			0.85           &           2512        &               6.6667     &   1.2494  &  0.0384  &  0.0825 \\ \hline
			\hline
		\end{tabular}
		\caption{Bias Comparison: $h_1(x) = (2x)*x^{2.75}$}
		\label{tb:adp_s2_bias}
	\end{center}
\end{table}

\begin{table}[!htb]
	\centering
	\begin{center}
		\begin{tabular}{|c|c|c|c|c|c|}
			\hline\hline
			$\alpha$-value & $d$-dimension & Smoothness Threshold & Plug-in & TF & Adaptive \\ \hline
			0.40             &          40         &               1.6667     &   0.0018  &  0.00003  &  0.0002\\ \hline
			0.45             &          63         &               1.8182     &   0.0028  &  0.00006  &  0.0002\\ \hline
			0.50             &         100         &               2.0000     &   0.0045  &  0.00007  &  0.0006\\ \hline
			0.55             &         158         &               2.2222     &   0.0073  &  0.00015  &  0.0012\\ \hline
			0.60             &         251         &               2.5000     &   0.0136  &  0.0008   &  0.0029\\ \hline
			0.65             &         398         &               2.8571     &   0.0246  &  0.0007   &  0.0045\\ \hline
			0.70             &         631         &               3.3333     &   0.0511  &  0.0021   &  0.0097\\ \hline
			0.75             &         1000        &               4.0000     &   0.1229  &  0.0036   &  0.0137\\ \hline
			0.80           &           1585        &               5.0000     &   0.3909  &  0.0066   &  0.0177\\ \hline
			0.85           &           2512        &               6.6667     &   1.2823  &  0.0091   &  0.0198\\ \hline
			\hline
		\end{tabular}
		\caption{Bias Comparison: $h_1(x) = (2x)^{3.75}$}
		\label{tb:adp_s3_bias}
	\end{center}
\end{table}

\begin{table}[!htb]
	\centering
	\begin{center}
		\begin{tabular}{|c|c|c|c|c|c|}
			\hline\hline
			$\alpha$-value & $d$-dimension & Smoothness Threshold & Plug-in & TF & Adaptive \\ \hline
			0.40             &          40         &               1.6667     &   0.0015  &  0.0015  &  0.0015\\ \hline
			0.45             &          63         &               1.8182     &   0.0024  &  0.0022  &  0.0022\\ \hline
			0.50             &         100         &               2.0000     &   0.0044  &  0.0038  &  0.0037\\ \hline
			0.55             &         158         &               2.2222     &   0.0091  &  0.0067  &  0.0065\\ \hline
			0.60             &         251         &               2.5000     &   0.0172  &  0.0105  &  0.0098\\ \hline
			0.65             &         398         &               2.8571     &   0.0542  &  0.0246  &  0.0216\\ \hline
			0.70             &         631         &               3.3333     &   0.1934  &  0.0578  &  0.0443\\ \hline
			0.75             &         1000        &               4.0000     &   0.7134  &  0.0885  &  0.0511\\ \hline
			0.80           &           1585        &               5.0000     &   24.170  &  0.3082  &  0.0951\\ \hline
			0.85           &           2512        &               6.6667     &   88.785  &  0.4164  &  0.2507\\ \hline
			\hline
		\end{tabular}
		\caption{MSE Comparison: $h_1(x) = (2x)^{2.75}$}
		\label{tb:adp_s2_mse}
	\end{center}
\end{table}

\begin{table}[!htb]
	\centering
	\begin{center}
		\begin{tabular}{|c|c|c|c|c|c|}
			\hline\hline
			$\alpha$-value & $d$-dimension & Smoothness Threshold & Plug-in  & TF  & Adaptive  \\ \hline
			0.40             &          40         &               1.6667     &   0.000121  &  0.000115  &  0.000114 \\ \hline
			0.45             &          63         &               1.8182     &   0.000236  &  0.000208  &  0.000204 \\ \hline
			0.50             &         100         &               2.0000     &   0.0004  &  0.0003  &  0.0003 \\ \hline
			0.55             &         158         &               2.2222     &   0.0010  &  0.0006  &  0.00056 \\ \hline
			0.60             &         251         &               2.5000     &   0.0035  &  0.0014  &  0.0011  \\ \hline
			0.65             &         398         &               2.8571     &   0.0153  &  0.0038  &  0.0026 \\ \hline
			0.70             &         631         &               3.3333     &   0.1024  &  0.0068  &  0.0030 \\ \hline
			0.75             &         1000        &               4.0000     &   1.5365  &  0.0318  &  0.0055 \\ \hline
			0.80           &           1585        &               5.0000     &   33.303  &  0.0692  &  0.0062 \\ \hline
			0.85           &           2512        &               6.6667     &   499.69  &  0.2806  &  0.1063 \\ \hline
			\hline
		\end{tabular}
		\caption{MSE Comparison: $h_2(x) = (2x)^{3.75}$}
		\label{tb:adp_s3_mse}
	\end{center}
\end{table}

\begin{table}[!htb]
	\centering
	\begin{center}
		\begin{tabular}{|c|c|c|c|c|c|}
			\hline\hline
			$\alpha$-value & $d$-dimension & Smoothness Threshold & Plug-in & TF &  Adaptive \\ \hline
			0.40             &          40         &               1.6667     &   0.0014  &  0.0015  &  0.0015 \\ \hline
			0.45             &          63         &               1.8182     &   0.0023  &  0.0022  &  0.0021 \\ \hline
			0.50             &         100         &               2.0000     &   0.0043  &  0.0038  &  0.0037 \\ \hline
			0.55             &         158         &               2.2222     &   0.0088  &  0.0067  &  0.0065 \\ \hline
			0.60             &         251         &               2.5000     &   0.0165  &  0.0104  &  0.0098  \\ \hline
			0.65             &         398         &               2.8571     &   0.0519  &  0.0246  &  0.0216 \\ \hline
			0.70             &         631         &               3.3333     &   0.1863  &  0.0578  &  0.0441 \\ \hline
			0.75             &         1000        &               4.0000     &   0.6905  &  0.0882  &  0.0496 \\ \hline
			0.80           &           1585        &               5.0000     &   24.030  &  0.3078  &  0.0919 \\ \hline
			0.85           &           2512        &               6.6667     &   88.069  &  0.4149  &  0.2450 \\ \hline
			\hline
		\end{tabular}
		\caption{Variance Comparison: $h_1(x) = (2x)^{2.75}$}
		\label{tb:adp_s2_var}
	\end{center}
\end{table}

\begin{table}[!htb]
	\centering
	\begin{center}
		\begin{tabular}{|c|c|c|c|c|c|}
			\hline\hline
			$\alpha$-value & $d$-dimension & Smoothness Threshold & Plug-in  & TF  &  Adaptive \\ \hline
			0.40             &          40         &               1.6667     &   0.00012  &  0.0001  &  0.0001 \\ \hline
			0.45             &          63         &               1.8182     &   0.00023  &  0.0002  &  0.0002 \\ \hline
			0.50             &         100         &               2.0000     &   0.00042  &  0.0003  &  0.0003\\ \hline
			0.55             &         158         &               2.2222     &   0.00095  &  0.0006  &  0.00055 \\ \hline
			0.60             &         251         &               2.5000     &   0.0033   &  0.0014   &  0.0011\\ \hline
			0.65             &         398         &               2.8571     &   0.0147   &  0.0026   &  0.0038 \\ \hline
			0.70             &         631         &               3.3333     &   0.0998   &  0.0070   &  0.0037 \\ \hline
			0.75             &         1000        &               4.0000     &   1.5214   &  0.0318   &  0.0054 \\ \hline
			0.80           &           1585        &               5.0000     &   33.152   &  0.1685   &  0.0059\\ \hline
			0.85           &           2512        &               6.6667     &   498.05   &  0.2805   &  0.1059\\ \hline
			\hline
		\end{tabular}
		\caption{Variance Comparison: $h_2(x) = (2x)^{3.75}$}
		\label{tb:adp_s3_var}
	\end{center}
\end{table}

\end{document}